\documentclass[sn-mathphys,Numbered]{sn-jnl}

\usepackage{graphicx}%
\usepackage{multirow}%
\usepackage{amsmath,amssymb,amsfonts}%
\usepackage{amsthm}%
\usepackage{mathrsfs}%
\usepackage[title]{appendix}%
\usepackage{xcolor}%
\usepackage{textcomp}%
\usepackage{manyfoot}%
\usepackage{booktabs}%
\usepackage{algorithm}%
\usepackage{algorithmicx}%
\usepackage{algpseudocode}%
\usepackage{listings}%

\usepackage{mathtools}

\theoremstyle{thmstyleone}%
\newtheorem{theorem}{Theorem}
\newtheorem{proposition}[theorem]{Proposition}%
\newtheorem{corollary}[theorem]{Corollary}
\newtheorem{lemma}[theorem]{Lemma}

\theoremstyle{thmstyletwo}%
\newtheorem{conjecture}{Conjecture}

\theoremstyle{thmstylethree}%

\newtheorem{innercustomthm}{Theorem}
\newenvironment{customthm}[1]
  {\renewcommand\theinnercustomthm{#1}\innercustomthm}
  {\endinnercustomthm}

\newtheorem{innercustomprop}{Proposition}
\newenvironment{customprop}[1]
  {\renewcommand\theinnercustomprop{#1}\innercustomprop}
  {\endinnercustomprop}

\newcommand{\Proj}{\text{Proj}}
\newcommand{\cdeg}{\text{c-}\deg}
\newcommand{\fqb}{\overline{\mathbb{F}}_q}
\newcommand{\pmfqb}{\PPP^m(\fqb)}
\newcommand{\amfqb}{\mathbb{A}^m(\fqb)}
\newcommand{\pmfq}{\PPP^m(\fq)}

\newcommand{\N}{\mathbb{N}}

\newcommand{\PPP}{\mathbb{P}}
\newcommand{\F}{\mathbb{F}}
\newcommand{\fq}{\F_q}

\newcommand{\HH}{\mathcal{H}}

\newcommand{\im}{\text{Im}}
\newcommand{\Spec}{\text{Spec}}

\newcommand{\A}{\mathbb{A}}

\begin{document}

\title[On a conjecture of Beelen, Datta and Ghorpade]{On a conjecture of Beelen, Datta and Ghorpade for the number of points of varieties over finite fields}


\author*[1]{\fnm{Deepesh} \sur{Singhal}\orcid{https://orcid.org/0000-0003-1317-5523}}\email{singhald@uci.edu}
\equalcont{These authors contributed equally to this work.}

\author[2]{\fnm{Yuxin} \sur{Lin}\orcid{https://orcid.org/0000-0001-9230-7728}}\email{yuxinlin@caltech.edu}
\equalcont{These authors contributed equally to this work.}

\affil*[1]{\orgdiv{Department of Mathematics}, \orgname{University of California Irvine}, \orgaddress{\street{Rowland Hall}, \city{Irvine}, \postcode{92697}, \state{California}, \country{USA}}}

\affil[2]{\orgdiv{Department of Mathematics}, \orgname{California Institute of Technology}, \orgaddress{\street{1200 E California Blvd}, \city{Pasadena}, \postcode{91125}, \state{California}, \country{USA}}}


\abstract{Consider a finite field $\fq$ and positive integers $d,m,r$ with $1\leq r\leq \binom{m+d}{d}$. Let $S_d(m)$ be the $\fq$ vector space of all homogeneous polynomials of degree $d$ in $X_0,\dots,X_m$. Let $e_r(d,m)$ be the maximum number of $\fq$-rational points in the vanishing set of $W$ as $W$ varies through all subspaces of $S_d(m)$ of dimension $r$. Beelen, Datta, and Ghorpade conjectured an exact formula of $e_r(d,m)$ when $q\geq d+1$. We prove that their conjectured formula is true when $q$ is sufficiently large in terms of $m$, $d$, $r$. The problem of determining $e_r(d,m)$ is equivalent to the problem of computing the $r^{th}$ generalized Hamming weight of the projective Reed-Muller code $PRM_q(d,m)$. It is also equivalent to the problem of determining the maximum number of points on sections of Veronese varieties by linear subvarieties of codimension $r$.}




\maketitle

\section{Introduction}
We begin by introducing some notation that we will use throughout this paper. Our first goal is to introduce the terminology necessary to state the conjecture of Beelen, Datta, and Ghorpade \cite{beelen2022combinatorial} that is the primary motivation for our work.

Let $\N$ be the set of nonnegative integers. We fix a finite field $\fq$. When we say that $X$ is an (affine) variety, we mean that $X$ is an irreducible projective (affine) variety defined over $\fqb$. When we say $X$ is an (affine) algebraic set, we mean that it is a projective (affine) algebraic set defined over $\fqb$. The set of $\fq$-rational points of $X$ will be denoted by $X(\fq)$. When we say $X$ is irreducible, we mean that it is irreducible over $\fqb$.
For a homogeneous ideal $I$ of $\fqb[X_0,\dots,X_{m+1}]$, we denote its vanishing set by $V(I)\subseteq \pmfqb$. For an ideal of $\fqb[x_1,\dots,x_m]$, we denote its zero set by $Z(I)\subseteq\amfqb$. We denote $\pi_m=|\pmfq|=\sum_{i=0}^m q^i$. Note that $\pi_0=1$ and we set $\pi_m=0$ for $m<0$.
For an algebraic set $X$, Lachaud and Rolland in \cite{lachaud2015number} define $\deg_i(X)$ to be the sum of the degrees of the irreducible components of $X$ having dimension $i$, and define $\cdeg(X)$ to be the sum of the degrees of all the irreducible components of $X$. In particular, if $k=\dim(X)$, then $\deg(X)=\deg_k(X)$ and $\cdeg(X)=\sum_{i=0}^{k}\deg_i(X)$.

Let $S(m)=\fq[x_0,\dots,x_m]$ and denote by $S_d(m)$ its $d^{th}$ graded component. Given a positive integer $r$ satisfying $r \leq \binom{m+d}{d}=\dim_{\fq}(S_d)$, Beelen, Datta, and Ghorpade in \cite{beelen2018maximum} define $e_r(d,m)$ as the maximal number of $\fq$-rational points among all projective algebraic sets defined by $r$ many degree $d$ homogeneous polynomials with $m+1$ variables. More precisely,
$$e_r(d,m)=\max\{|V(F_1,\dots,F_r)(\fq)| \mid F_1,\dots, F_r \in S_d(m)  \text{ are linearly independent}\}.$$
Let $T(m)=\fq[x_1, \dots, x_m]$ and $T_{\leq d}(m)$ be the degree $\leq d$ part. They analogously define $e_r^{\A}(d,m)$ as the maximal number of zeros of an affine variety defined by $r$ many polynomials of degree at most $d$ and $m$ variables. That is,
$$e_r^{\A}(d,m)=\max\{|Z(f_1,\dots,f_r)(\fq)| \mid f_1,\dots, f_r \in T_{\leq d}(m) \text{ are linearly independent}\}.$$
Beelen, Datta and Ghorpade in \cite{beelen2018maximum} define the set $\Omega(d,m)$,
$$\Omega(d,m)=\Big\{(\gamma_1,\dots,\gamma_{m+1})\in \N^{m+1}\mid \sum_{i=1}^{m+1} \gamma_i=d\Big\}.$$
For $1\leq r\leq \tbinom{m+d}{d}$, let $\omega_r(d,m)=(\beta_1,\dots,\beta_{m+1})$ be the $r^{th}$ largest element in $\Omega(d,m)$ under the lexicographical ordering. Then, they define the quantity
$$H_r(d,m)=\sum_{i=1}^m \beta_iq^{m-i}.$$
When $r=0$, set $H_0(d,m)=q^m$. When $r>\tbinom{m+d}{d}$, set $H_r(d,m)=0$.

Heijnen and Pellikaan in \cite{heijnen1998generalized} obtain an exact formula for $e_r^{\A}(d,m)$.
\begin{theorem}\cite[Theorem 6.8]{heijnen1998generalized}\label{Affine result}
Given $m,d\geq 1$, $1 \leq r \leq \binom{m+d}{d}$ and $q\geq d+1$, we have
$$e_r^{\A}(d,m)=H_r(d,m).$$
\end{theorem}
In fact, Heijnen and Pellikaan find a formula for $e_r^{\A}(d,m)$, without the requirement that $q\geq d+1$. This more general formula becomes $H_r(d,m)$ when $q\geq d+1$.
Boguslavsky and Tsfasman in \cite{boguslavsky1997number} conjectured an exact formula for $e_r(d,m)$.

\begin{conjecture}[Boguslavsky-Tsfasman conjecture]\label{TBC}\cite[Conjecture 3, Corollary 5]{boguslavsky1997number}\\
Suppose that we are given $m,d\geq 1$, $1 \leq r \leq \binom{m+d}{d}$ and $q\geq d+1$. Denote $w_r(d,m)=(\beta_1,\dots,\beta_{m+1})$ and $l=\min\{i\mid \beta_i\neq 0\}$. Then, for $q\geq d+1$, we have
$$e_r(d,m)=\sum_{i=l}^{m}\beta_{i}(\pi_{m-i}-\pi_{m-i-l}) +\pi_{m-2l}.$$
\end{conjecture}

The case $r=1$ of Conjecture~\ref{TBC} was previously shown by Serre in \cite{serre1991lettre} and S\o{}rensen in \cite{sorensen1991projective}. The case $r=2$ was proven by Boguslavsky in \cite{boguslavsky1997number}. Zanella in \cite{zanella1998linear} explicitly computed $e_r(2,m)$. Datta and Ghorpade in \cite{datta2017number} prove Conjecture~\ref{TBC} for $r\leq m+1$.
However, Datta and Ghorpade in \cite{datta2015conjecture} disproved Conjecture~\ref{TBC} for $r> m+1$ and proposed a new conjectured formula for $e_r(d,m)$ valid when $1\leq r\leq \binom{m+d-1}{d-1}$.

\begin{conjecture}[Incomplete Datta-Ghorpade conjecture]\cite[Conjecture 6.6]{datta2017number}\label{conjecture dqr}\\
Given $m,d\geq 1$, $1 \leq r \leq \binom{m+d-1}{d-1}$ and $q\geq d+1$, we have
$$e_r(d,m)=H_r(d-1,m)+\pi_{m-1}.$$
\end{conjecture}

It is seen that for $r\leq m+1$, Conjecture \ref{TBC} and Conjecture \ref{conjecture dqr} give the same formula, but they differ for $r>m+1$. Moreover, in \cite{datta2017number} Beelen, Datta, and Ghorpade prove Conjecture~\ref{conjecture dqr} for the case $d=1$ and also for the case $m=1$. In \cite{beelen2018maximum}, they prove Conjecture~\ref{conjecture dqr} for $r\leq \binom{m+2}{2}$, and in \cite{beelen2022combinatorial} they extended the conjecture to cover all values of $r$, that is, $1\leq r\leq \tbinom{m+d}{d}$.

\begin{conjecture}[Complete Beelen-Datta-Ghorpade conjecture]\cite[Equation 7]{beelen2022combinatorial}\label{Conj: complete GDC}\\
Suppose that we are given $m,d\geq 1$, $1 \leq r \leq \binom{m+d}{d}$ and $q\geq d+1$. Let $l$ be the unique integer such that $1\leq l\leq m+1$ and
$$\tbinom{m+d}{d}-\tbinom{m+d+1-l}{d} <r \leq \tbinom{m+d}{d}-\tbinom{m+d-l}{d}.$$
Let $j=r-\binom{m+d}{d}+\binom{m+d+1-l}{d}$, so $0< j\leq \binom{m+d-l}{d-1}$.
Then we have
$$e_r(d,m)=H_j(d-1,m-l+1) +\pi_{m-l}.$$
\end{conjecture}
Note that Conjecture~\ref{conjecture dqr} covers the values of $r$ for which $l=1$.
Conjecture~\ref{Conj: complete GDC} was proven for $\binom{m+d}{d}-d\leq r\leq \binom{m+d}{d}$ by Datta and Ghorpade in \cite{datta2017remarks}. This range of $r$ corresponds to $l\in \{m,m+1\}$.
Beelen, Datta, and Ghropade in \cite{beelen2022combinatorial} show that their conjectured formula is a lower bound for $e_r(d,m)$.

\begin{proposition}\cite[Theorem 2.3]{beelen2022combinatorial}\label{Lower bound for erdm}\\
Suppose that we are given $1 \leq r \leq \binom{m+d}{d}$ and $q\geq d+1$. Furthermore, suppose that
$$\tbinom{m+d}{d}-\tbinom{m+d+1-l}{d} <r \leq \tbinom{m+d}{d}-\tbinom{m+d-l}{d},$$
and $j=r-\binom{m+d}{d}+\binom{m+d+1-l}{d}$. Then we have
$$e_r(d,m)\geq H_j(d-1,m-l+1) +\pi_{m-l}.$$
\end{proposition}
The previously proven cases of Conjecture \ref{Conj: complete GDC} cover very specific values for $(m,d,r)$, but allow for any $q \geq d+1$.
In this paper our goal is to prove Conjecture~\ref{Conj: complete GDC} for any possible $(m,d,r)$, but for sufficiently large $q$, and we give an effective bound on $q$ in terms of $(m,d,r)$.
We would also like to point out that the proof of Theorem~\ref{Affine result} and the proofs of existing cases of Conjecture~\ref{Conj: complete GDC} are purely combinatorial. However, we will make use of tools from algebraic geometry in our work.

Notice that in Conjecture~\ref{Conj: complete GDC}, the authors divide the range $1\leq r\leq \binom{m+d}{d}$ into smaller sub-ranges based on the value of $l$, which can vary as $1\leq l\leq m+1$. It turns out that if $r$ belongs to the range corresponding to $l$, then the conjectured formula of $e_r(d,m)$ is a polynomial in $q$ of degree $m-l$. In Theorem~\ref{Thm: complete GDC conjecture}, we further divide these sub-ranges based on values of $l$ and $c$ that can vary as $1\leq l\leq m$ and $1\leq c\leq d$. If $r$ belongs to the range corresponding to $l,c$, then the conjectured polynomial of $e_r(d,m)$ has degree $m-l$ and leading coefficient $c$.

Our main result in this paper is the following.

\begin{theorem}\label{Thm: complete GDC conjecture}
Suppose that we have $m,d\geq 1$ and $1\leq r\leq \binom{m+d}{d}$.
Let $l$ and $c$ be the unique integers such that $1\leq l\leq m$, $1\leq c\leq d$ and
\begin{align*}
\tbinom{m+d}{d}&-\tbinom{m+d+1-l}{d}+\tbinom{m+d-l-c}{d-c-1} <r \leq \tbinom{m+d}{d}-\tbinom{m+d+1-l}{d}+\tbinom{m+d+1-l-c}{d-c}.
\end{align*}
If
$$q\geq \max\Big\{2(m-l+1)c^2+1,\;8\frac{d^{l+1}}{c},\; 164 c^{14/3}\Big\},$$
then we have
$$e_r(d,m)= H_{r-\binom{m+d}{d}+\binom{m+d+1-l}{d}}(d-1,m-l+1)+\pi_{m-l}.$$
\end{theorem}

Since Proposition~\ref{Lower bound for erdm} shows that the conjectured formula is lower bound for $e_r(d,m)$, we only need to show that it is an upper bound.
We want to show that (for large $q$), given any linearly independent polynomials $F_1,\dots, F_r\in S_d(m)$, $|V(F_1,\dots,F_r)(\fq)|$ is at most the conjectured formula for $e_r(d,m)$. This will complete the proof of Theorem~\ref{Thm: complete GDC conjecture}.

We start by studying the dimension and degree of $V(F_1,\dots ,F_r)$ for linearly independent $F_i\in S_d(m)$ and show the following.

\begin{proposition}\label{dimension}
Suppose that we have $1\leq l\leq m$ and $\binom{m+d}{d}-\binom{m+d+1-l}{d} <r \leq \binom{m+d}{d}$. Given $F_1, \dots, F_r \in S_d(m)$ that are linearly independent, we have
$$\dim(V(F_1, \dots, F_r)) \leq m-l.$$
\end{proposition}
\begin{proposition}\label{degree}
Suppose that we have $1\leq l\leq m$, $1\leq c\leq d$ and
$$\tbinom{m+d}{d}-\tbinom{m+d+1-l}{d}+\tbinom{m+d-l-c}{d-c-1} <r \leq \tbinom{m+d}{d}.$$
Then given $F_1, \dots, F_r \in S_d(m)$ that are linearly independent, we have:
$$\deg_{m-l}(V(F_1, \dots, F_r)) \leq c.$$
\end{proposition}
Proposition~\ref{dimension} and Proposition~\ref{degree} are proven in Section~\ref{sec: two}.
Note that the dimension and degree of $X$ only give us information about the highest-dimensional components of $X$.
Next, we find a bound on the number of $\fq$-rational points on components of $X$ of dimension smaller than $m-l$. These will be referred to as low-dimensional components of $X$.
\begin{proposition}\label{Prop: Components of codim geq l}
Suppose $X$ is the vanishing set of a collection of homogeneous polynomials in $\fqb[X_0,\dots,X_m]$, each having degree at most $d$. Let $Y$ be the union of the irreducible components of $X$ that have dimensions at most $k$. Then, for $q\geq d$, we have
$$|Y(\fq)|\leq d^{m-k} \pi_{k}.$$
\end{proposition}
Proposition \ref{Prop: Components of codim geq l} is proven in Section~\ref{sec: lower dim components} using intersection theory. The bound provided by Proposition~\ref{Prop: Components of codim geq l} is not sharp and is independent of the number of polynomials.

In Section~\ref{sec: case when X does not contain a linear subspace}, we prove the case of Theorem \ref{Thm: complete GDC conjecture} where $X$ does not contain a linear subspace of dimension $(m-l)$ defined over $\fq$ (recall that $l$ is determined by which range $r$ is in). We use results of Cafure and Matera from \cite{cafure2006improved} that bound the number of $\fq$-rational points of a variety in terms of its dimension and degree. Along with Proposition~\ref{dimension} and Proposition~\ref{degree} we obtain a bound on the number of $\fq$-rational points on components of $X$ of dimension $(m-l)$. We use Proposition~\ref{Prop: Components of codim geq l} to bound the number of $\fq$-rational points on the lower-dimensional components of $X$. The assumption that $X$ does not have a linear subspace of dimension $m-l$ defined over $\fq$ allows us to get a very good bound on the number of $\fq$-rational points on the $(m-l)$-dimensional components of $X$. Assuming that $q$ is sufficiently large, we can prove the case of Theorem \ref{Thm: complete GDC conjecture} where $X$ does not contain a linear subspace of dimension $m-l$ defined over $\fq$.

In Section~\ref{sec: if X contains a linear subspace}, we prove the case of Theorem~\ref{Thm: complete GDC conjecture} when $X$ contains a linear subspace of dimension $(m-l)$ defined over $\fq$. In this case, equality can actually hold, so we need to be very precise in dealing with components of all dimensions. We consider the complement of the $(m-l)$-dimensional linear subspace in $X$. We are able to divide this complement into a number of affine algebraic sets and apply Theorem~\ref{Affine result} to them. This leads to very precise estimates on components of all dimensions. With a technical combinatorial argument, we complete the proof of the case of Theorem~\ref{Thm: complete GDC conjecture} when $X$ contains a linear subspace of dimension $(m-l)$ defined over $\fq$.
Section~\ref{sec: case when X does not contain a linear subspace} and Section~\ref{sec: if X contains a linear subspace} together complete the proof of Theorem~\ref{Thm: complete GDC conjecture}.

We also show that Conjecture~\ref{conjecture dqr} is true for $q\geq (d-1)^2$. This corresponds to the case of Conjecture \ref{Conj: complete GDC} where $l=1$. This is done in Section~\ref{sec: proof for incomplete gdc}.

\begin{theorem}\label{Thm l=1 with e}
Suppose $m\geq 2$, $d\geq 2$, $0\leq e\leq d-2$ and $\binom{m+e}{e}< r \leq \binom{m+e+1}{e+1}$.
If $q\geq \max\{d+e+\frac{e^2-1}{d-(e+1)}, d-1+e^2-e\}$, then we have
$$e_r(d,m)= H_r(d-1,m)+\pi_{m-1}.$$
\end{theorem}
By varying the range of $e$, Theorem \ref{Thm l=1 with e} leads to the following corollary:
\begin{corollary}
Suppose that we are given $m,d\geq 1$, $1\leq r\leq \binom{m+d-1}{d-1}$. Then for $q\geq (d-1)^2$, we have
$$e_r(d,m)=H_r(d-1,m)+\pi_{m-1}.$$    
\end{corollary}

We defer the proofs of many technical lemmas to Appendix~\ref{Sec: Appendix}.

\section{Dimension and degree}\label{sec: two}

When we write that $X$ is a projective subscheme of $\mathbb{P}^m$, we mean that $X$ is of the form $\Proj(\fqb[x_0,\dots,x_m]/I)$ for some homogeneous ideal $I$. If $X=\Proj(\fqb[x_0,\dots,x_m]/I)$, then $I(X)$ is the saturation of $I$ and $\Proj(\fqb[x_0,\dots,x_m]/I)=\Proj(\fqb[x_0,\dots,x_m]/I(X))$. See \cite[Section 5]{Gathmann2002} for a reference.
The coordinate ring of $X$ is the graded ring $S(X)=\fqb[x_0,\dots,x_m]/I(X)$ and its degree $t$ part is denoted as $S_t(X)$. We have
\[\dim_{\fqb}((\fqb[x_0,\dots,x_m]/I)_t) \geq \dim_{\fqb}(S_t(X)).\]

In this section, our goal is to prove Proposition~\ref{dimension} and Proposition~\ref{degree}.
These two propositions are saying that if $\dim(I(X)_{t})$ is big, then the dimension and degree of $X$ are small. Note that $\dim(I(X)_{t})$ being big is the same as $\dim(S(X)_{t})$ being small. Therefore, the contrapositive statement is that if the dimension and degree of $X$ are big, then $\dim(S(X)_{t})$ is also big.
We will thus consider projective subschemes of $\pmfqb$ of given dimension and degree, and find lower bounds for the Hilbert function in terms of the dimension and degree. We start by considering the zero-dimensional case.

\begin{lemma}\label{base case r}
Let $X$ be a zero-dimensional projective subscheme of $\pmfqb$ of degree $c$.
\begin{enumerate}
    \item If $0\leq t \leq c-1$, then $\dim_{\fqb}(S_{t}(X)) \geq t+1$.
    \item If $t \geq c-1$, then $\dim_{\fqb}(S_t(X))=c$.
\end{enumerate}    
\end{lemma}
\begin{proof}
There is a hyperplane $\HH \subseteq \pmfqb$ such that $\HH$ does not contain any point of $X$ (since $\dim(X)=0$). By a linear change of coordinates, we can assume that $\HH=V(x_0)$. Then $X=X\setminus \HH$, and $X$ is a zero-dimensional affine scheme, $X=\Spec(R)$ with $R=\fqb[x_1, \dots, x_n]/I$, where $I$ is the dehomogenization of $I(X)$ with respect to $x_0$. In addition, $I(X)$ is the homogenization of $I$ with respect to $x_0$ (see \cite[Lemma 6.1.4]{Gathmann2002}). Therefore, $\dim_{\fqb}(I(X)_d)=\dim_{\fqb}(I_{\leq d})$ as $\fqb$ vector spaces. Let $R_{\leq t}=\fqb[x_1, \dots, x_m]_{\leq t}/I_{\leq t}$.
We have an isomorphism of $\fqb$ vector spaces: $S_t(X) \to R_{\leq t}$ given by homogenization and dehomogenization with respect to $x_0$.  

Now, suppose that for some $0\leq t\leq c-1$, we have $\dim(S_t) \leq t$. Then $\dim(R_{\leq t}) \leq t$ and therefore there exists some $0 \leq i \leq t-1$ such that $\dim(R_{\leq i})=\dim(R_{\leq i+1})$. This means that $R_{\leq i}=R_{\leq i+1}$.
We want to show by induction that for any $j \geq i, R_{\leq j}=R_{\leq i}$. Suppose that this is true for $j-1$. Consider a monomial $M_j$ of degree $j$ and write $M_j=M_{i+1}M_{j-i-1}$, where $M_{i+1}$ is a monomial of degree $i+1$. Since $R_{\leq i}=R_{\leq i+1}$, we know that $M_{i+1}$ is congruent mod $I$ to some polynomial of degree $\leq i$ and thus $M_j$ is congruent mod $I$ to some polynomial of degree $\leq j-1$. Since this holds for each monomial $M_j$ of degree $j$, we conclude that $R_{\leq j}=R_{\leq j-1}=R_{\leq i}$. This finishes the inductive step, and hence $R_{\leq j}=R_{\leq i}$ for every $j\geq i$.
Since $X$ is a zero-dimensional subscheme, $\deg(X)=\dim(R)$ is finite and equal to $\dim(R_{\leq d})$ for sufficiently large $d$. Therefore, $c=\deg(X)=\dim(R)=\dim(R_{\leq i}) = \dim(R_{\leq t}) \leq t$. This contradicts the fact that $t\leq c-1$. Therefore, we see that for every $0\leq t\leq c-1$, we have $\dim_{\fqb}(S_t(X))\geq t+1$.

In particular, we have $\dim(S_{c-1}(X))\geq c$. Therefore, for $t\geq c-1$, we have
$$c\leq \dim(R_{\leq c-1})\leq \dim(R_{\leq t})\leq \dim(R)=c.$$
This means that $\dim(S_t(X))=c$.
\end{proof}

Next, we consider projective subschemes of $\pmfqb$ of arbitrary dimension. We will induct on the dimension of the subscheme, and Lemma~\ref{base case r} will serve as the base case.

\begin{proposition}\label{Prop: dim deg bound on S_t}
Let $X$ be a projective subscheme of $\pmfqb$ with $\dim(X)=k$ and $\deg(X)=c$. Then,
\begin{enumerate}
    \item For $t \leq c-1$, we have $\dim(S_t(X)) \geq \binom{t+k+1}{k+1}$.
    \item For $t \geq c$, we have $\dim(S_t(X)) \geq \binom{t+k+1}{k+1}-\binom{t+k+1-c}{k+1}$.
\end{enumerate}
\end{proposition}
\begin{proof}
Note that if we prove this result for equidimensional subschemes, then it will automatically follow for all subschemes. This is because we can start with an arbitrary subscheme $X$ and let $X_1$ be the union of the components of $X$ of dimension $\dim(X)$. Note that $\dim(X_1)=\dim(X)$, $\deg(X_1)=\deg(X)$, and $\dim(S_t(X))\geq \dim(S_t(X_1))$. Therefore, if the result is proved for $X_1$, then it will automatically follow for $X$.

Therefore, we will prove this for equidimensional subschemes by induction on $k$, the base case was proven in Lemma \ref{base case r}.
Now suppose that the result is known for $k-1$. Let $X$ be an equidimensional projective subscheme of $\pmfqb$ with $\dim(X)=k$ and $\deg(X)=c$. We choose a hyperplane $\HH$ that does not contain any irreducible component of $X$, then $X \cap \HH$ has degree $c$ and it is equidimensional with $\dim(X \cap \HH)=k-1$. After linear change of variables, assume $\HH=V(x_0)$.
Denote $h_X(t)=\dim(S_t(X))$ and $h_{X\cap \HH}(t)=\dim(S_t(X\cap \HH))$.
Since no irreducible component of $X$ is in $\HH=V(x_0)$, we have the exact sequence 
$$0 \to \fqb[x_0, \dots, x_m]/I(X) \xhookrightarrow{x_0} \fqb[x_0, \dots, x_m]/I(X)  \to \fqb[x_0, \dots, x_m]/\langle I(X)+(x_0)\rangle \to 0.$$
This means that
$$h_X(t)-h_X(t-1) =\dim\Big(\big(\fqb[x_0, \dots, x_m]/\langle I(X)+(x_0)\rangle\big)_t\Big) \geq h_{X \cap \HH}(t).$$
Thus $h_X(t)\geq 1+\sum_{j=1}^t h_{X \cap \HH}(j)$.
By induction hypothesis, we know that 
\begin{itemize}
    \item For $j \leq c-1$, we have $h_{X\cap \HH}(j) \geq \binom{j+k}{k}$.
    \item For $j \geq c$, we have $h_{X\cap \HH}(j) \geq \binom{j+k}{k}-\binom{j+k-c}{k}$.
\end{itemize}

Next, consider some $t \leq c-1$.
For $j \leq t$, we have $h_{X \cap \HH}(j)\geq \binom{j+k}{k}$. Therefore, \[h_X(t) \geq \sum_{j=0}^t \tbinom{j+k}{k}=\tbinom{t+k+1}{k+1}.\]

Next consider some $t \geq c$. We have
\begin{align*}
    h_X(t)&\geq 1+\sum_{j=1}^th_{X \cap \HH}(j) 
     \geq 1+ \sum_{j=1}^{c-1}\tbinom{j+k}{k}+\sum_{j=c}^t\tbinom{j+k}{k}-\tbinom{j+k-c}{k}\\
    &=\sum_{j=0}^{t}\tbinom{j+k}{k}- \sum_{i=0}^{t-c} \tbinom{k+i}{k}
    =\tbinom{t+k+1}{k+1}-\tbinom{t+k+1-c}{k+1}.\qedhere
\end{align*}
\end{proof}

Note that in both cases (whether $t\leq c-1$ or $t\geq c$), we have
\[\dim(S_t(X))\geq \tbinom{t+k+1}{k+1}-\tbinom{t+k+1-c}{k+1}.\]
We are now ready to prove Proposition~\ref{dimension} and Proposition~\ref{degree}. Our proofs will rely on Proposition~\ref{Prop: dim deg bound on S_t}.

\begin{customprop}{\ref{dimension}}
Suppose that we have $1\leq l\leq m$ and $\binom{m+d}{d}-\binom{m+d+1-l}{d} <r \leq \binom{m+d}{d}$. Given $F_1, \dots, F_r \in S_d(m)$ that are linearly independent, we have:
$$\dim(V(F_1, \dots, F_r)) \leq m-l.$$
\end{customprop}
\begin{proof}[Proof of Proposition \ref{dimension}]
Let $X=V(F_1,\dots,F_r)$. If $X=\emptyset$, then we are done, so assume that $X\neq\emptyset$.
Since $F_1,\dots,F_r\in I(X)$, we have
$$\dim(S_{d}(X))\leq \tbinom{m+d}{d}-r< \tbinom{d+m-l+1}{d}.$$
Let $k=\dim(X)$ and $c=\deg(X)$. We have $c\geq 1$, since $X\neq\emptyset$. Then by Proposition~\ref{Prop: dim deg bound on S_t}, we have
\begin{align*}
\dim(S_d(X))&\geq \tbinom{d+k+1}{k+1}-\tbinom{d+k+1-c}{k+1}
\geq \tbinom{d+k+1}{k+1}-\tbinom{d+k+1-1}{k+1}=\tbinom{d+k}{k}=\tbinom{d+k}{d}.
\end{align*}
Now, we have
$$ \tbinom{d+k}{d}\leq \dim(S_d(X)) < \tbinom{d+m-l+1}{d},$$
which means that $k< m-l+1$, that is, $\dim(X)\leq m-l$.
\end{proof}

The proof of Proposition~\ref{degree} is similar.

\begin{customprop}{\ref{degree}}
Suppose that we have $1\leq l\leq m$, $1\leq c\leq d$ and
$$\tbinom{m+d}{d}-\tbinom{m+d+1-l}{d}+\tbinom{m+d-l-c}{d-c-1} <r \leq \tbinom{m+d}{d}.$$
Then given $F_1, \dots, F_r \in S_d(m)$ that are linearly independent, we have:
$$\deg_{m-l}(V(F_1, \dots, F_r)) \leq c.$$
\end{customprop}
\begin{proof}[Proof of Proposition \ref{degree}]
Let $X=V(F_1,\dots,F_r)$. Since $F_1,\dots,F_r\in I(X)$, we have
\begin{align*}
\dim(S_{d}(X))&\leq \tbinom{m+d}{d}-r< \tbinom{d+m-l+1}{d}-\tbinom{m+d-l-c}{d-c-1}
=\tbinom{d+m-l+1}{d}-\tbinom{m+d-l-c}{m-l+1}.
\end{align*}
Now, we know from Proposition \ref{dimension} that $\dim(X)\leq m-l$. If $\dim(X)<m-l$, then $\deg_{m-l}(X)=0$ and we are done. Therefore, assume $\dim(X)=m-l$.
Let $c_1=\deg(X)$. Then by Proposition \ref{Prop: dim deg bound on S_t}, we have
\begin{align*}
\dim(S_d(X))\geq \tbinom{d+m-l+1}{m-l+1}-\tbinom{d+m-l+1-c_1}{m-l+1}.
\end{align*}
Now, we have
$$ \tbinom{m+d-l-c}{m-l+1} < \tbinom{d+m-l+1-c_1}{m-l+1},$$
which means that $m+d-l-c< d+m-l+1-c_1$, that is, $\deg(X)=c_1\leq c$.
\end{proof}

\section{Lower dimensional Components}\label{sec: lower dim components}
In this section, our goal is to prove Proposition~\ref{Prop: Components of codim geq l}. In Proposition~\ref{Prop: Components of codim geq l}, $X$ is the intersection of the vanishing sets of homogeneous polynomials of degree at most $d$. We want to prove an upper bound on the total number of $\fq$-rational points on all low-dimensional components of $X$ and the upper bound does not depend on the number of polynomials.
We recall the following proposition by Lachaud and Rolland, which gives the relation between the number of $\fq$-rational points of an algebraic set in terms of its degree and dimension.
\begin{proposition}\cite[Theorem 2.1]{lachaud2015number}\label{Lachaud bound}
If $Y$ is an algebraic set of dimension $k$, then we have
$$|Y(\fq)|\leq \sum_{i=0}^{k}\deg_i(Y) \pi_i.$$
\end{proposition}
As we increase $r$, that is, intersect more hypersurfaces, we expect the dimension of the components to decrease and their degrees to increase. We will formalize this and show that a certain weighted sum of degrees of the components of $X$ remains bounded.
Using the bound on the weighted sum of the degrees of the components of the lower dimension of $X$ along with Proposition~\ref{Lachaud bound}, we will bound the number of $\fq$-rational points on components of $X$ with small dimension.

For an algebraic set $X\subseteq\pmfqb$, recall that $\deg_j(X)$ is the sum of the degrees of all the $j$ dimensional irreducible components of $X$. First, we consider how the degree of an equidimensional algebraic set changes when intersected with a degree $d$ hypersurface.
\begin{lemma}\label{degree intersection}
If $X\subseteq\pmfqb$ is an equidimensional algebraic set of dimension $k$ and $F$ is a homogeneous polynomial of degree at most $d$, then we have
$$\deg_{k-1}(X\cap V(F))+ d\deg_{k}(X\cap V(F))\leq d \deg(X).$$
Moreover, all irreducible components of $X\cap V(F)$ have dimension $k$ or $k-1$.
\end{lemma}
\begin{proof}
Let $d_1=\deg(F)\leq d$. Let $X_1,\dots,X_t$ be the irreducible components of $X$ that are contained in $V(F)$ and $Y_1,\dots, Y_s$ be the irreducible components of $X$ that are not contained in $V(F)$. 

First, suppose $k\geq 1$.
Then $X_i\cap V(F)=X_i$ and $Y_i\cap V(F)$ is equidimensional of dimension $k-1$ with degree $d_1 \deg(Y_i)$ (see \cite[Theorem 6.2.1]{Gathmann2002}). This means that $\deg_{k}(X\cap V(F))=\sum\deg(X_i)$ and $\deg_{k-1}(X\cap V(F))=d_1\sum\deg(Y_i)$. Therefore
$$\deg_{k-1}(X\cap V(F))+d_1\deg_k(X\cap V(F))=d_1 \deg(X).$$

Next, if $k=0$, then $X\cap V(F)=\bigcup X_i$. Therefore, $\deg_k(X\cap V(F))=\sum \deg(X_i)\leq \deg(X)$ and $\deg_{k-1}(X\cap V(F))=0$.
\end{proof}

Now, we consider the $r$-fold intersection of $V(F_i)$. We show that the 'weighted sum' of degrees from different dimensional components is bounded.
\begin{proposition}\label{Prop: bound on the weighted sum}
Suppose $F_1,\dots,F_r\in \fqb[x_0,\dots,x_m]$ are homogeneous polynomials of degree at most $d$ and let $X=V(F_1,\dots,F_r)$. Then we have
$$\sum_{j=1}^{m} d^{-j}\deg_{m-j}(X) \leq 1.$$
\end{proposition}
\begin{proof}
We denote $X_k=V(F_1,\dots,F_k)$ and $\alpha(j,k)=\deg_{m-j}(X_k)$. We will prove the result by induction on $k$ ($1\leq k\leq r$) that
$$\sum_{j=1}^{m}d^{-j}\alpha(j,k)\leq 1.$$
When $k=1$, we have $\alpha(1,1)=\deg_{m-1}(V(F_1)) = \deg(F_1)\leq d$ and $\alpha(1,j)=0$ for $j>1$. So we are done with the base case.

Next, suppose that we know the conclusion for $k-1$. Let $\Gamma_j$ be the union of the codimension $j$ irreducible components of $X_{k-1}$. We know by Lemma \ref{degree intersection} that for each $1\leq j\leq k-1$,
$$\deg_{m-j-1}(\Gamma_j\cap V(F_k))+d\deg_{m-j}(\Gamma_j\cap V(F_k)) \leq d\deg(\Gamma_j) =d\alpha(j,k-1).$$
Now, clearly $\alpha(j,k)=\deg_{m-j}(\Gamma_j\cap V(F_k))+\deg_{m-j}(\Gamma_{j-1}\cap V(F_k))$. Therefore,
\begin{align*}
\sum_{j=1}^m d^{-j}\alpha(j,k) 
&=\sum_{j=1}^m d^{-j}\Big(\deg_{m-j}(\Gamma_j\cap V(F_k))+\deg_{m-j}(\Gamma_{j-1}\cap V(F_k))\Big)\\
&=\sum_{j=1}^{m} \Big(d^{-j}\deg_{m-j}(\Gamma_j\cap V(F_k)) +d^{-j-1}\deg_{m-j-1}(\Gamma_{j}\cap V(F_k)) \Big)\\
&=\sum_{j=1}^{m} d^{-j-1}\Big(d\deg_{m-j}(\Gamma_j\cap V(F_k)) +\deg_{m-j-1}(\Gamma_{j}\cap V(F_k)) \Big) \\
&\leq \sum_{j=1}^{m}d^{-j-1} d\alpha(j,k-1) = \sum_{j=1}^{m}d^{-j} \alpha(j,k-1) \leq 1.
\end{align*}
This completes the inductive step.
Now, for $k=r$, we see that
\[\sum_{j=1}^{m}d^{r-j} \deg_{m-j}(X)\leq 1. \qedhere\]
\end{proof}

We now apply Proposition~\ref{Prop: bound on the weighted sum} and Proposition~\ref{Lachaud bound} to our setting and obtain an upper bound on the number of $\fq$-rational points on components of $X$ of small dimension.
\begin{customprop}{\ref{Prop: Components of codim geq l}}
Suppose $X$ is the vanishing set of a collection of homogeneous polynomials in $\fqb[X_0,\dots,X_m]$, each having degree at most $d$. Let $Y$ be the union of the irreducible components of $X$ that have dimensions at most $k$. Then, for $q\geq d$, we have
$$|Y(\fq)|\leq d^{m-k} \pi_{k}.$$
\end{customprop}
\begin{proof}
First, we will show that for $j\geq m-k$, we have $d^j\pi_{m-j}\leq d^{m-k}\pi_{k}$. This is because
\[d^{j+k-m} \pi_{m-j} = d^{j+k-m} \sum_{i=0}^{m-j} q^i
\leq q^{j+k-m} \sum_{i=0}^{m-j} q^i \leq \sum_{i=0}^{k} q^i =\pi_{k}.\]
By Proposition~\ref{Lachaud bound} and Proposition~\ref{Prop: bound on the weighted sum}, we see that
\begin{align*}
|Y(\fq)|&\leq \sum_{j=m-k}^{m}\deg_{m-j}(X)\pi_{m-j} 
=\sum_{j=m-k}^{m}d^{-j}\deg_{m-j}(X) d^{j}\pi_{m-j}\\
&\leq d^{m-k} \pi_{k} \sum_{j=l}^{m} d^{-j}\deg_{m-j}(X) \leq d^{m-k} \pi_{k}. \qedhere
\end{align*}
\end{proof}

\section{If $X$ does not contain a linear subspace}\label{sec: case when X does not contain a linear subspace}
In this section, we prove the case of Theorem~\ref{Thm: complete GDC conjecture}, when $X$ does not contain a $(m-l)$-dimensional linear subspace defined over $\fq$.
In Proposition~\ref{dimension} and Proposition~\ref{degree}, we have bounded the dimension and degree of $X$. We want to turn this into an upper bound on the number of $\fq$-rational points on the $(m-l)$-dimensional components of $X$.
Proposition~\ref{Lachaud bound} tells us that this is at most $c\pi_{m-l}$, however, we need a tighter bound. We will remind the reader that we call an algebraic set irreducible if it is irreducible over $\fqb$ and varieties are irreducible.

We recall the following result of Cafure and Matera that bounds the number of $\fq$-rational points of an affine variety in terms of its degree and dimension. 
\begin{theorem}\cite[Theorem 7.1]{cafure2006improved}\label{thm: affine variety bound}\\
Let $X$ be an affine variety defined over $\fq$ of dimension $k>0$ and degree $\delta$. If $q > 2(k + 1)\delta^2$, then the following estimate holds:
$$||X(\fq)| - q^k| \leq (\delta - 1)(\delta - 2)q^{k- 1/2} + 5\delta^{13/3} q^{k-1}.$$
\end{theorem}
Note that the main term is $q^k$, regardless of the degree of the affine variety.
So, if $X$ consisted of $c$ irreducible components, each of degree $1$, then the main term can still be $c q^k$. However, if it has components of degree $\delta>1$, then this bound on those components is much tighter than $\delta \pi_k$.

We prove a projective version of this result by applying a counting argument.

\begin{corollary}\label{Cor bound}
Let $X$ be a variety defined over $\fq$ of dimension $k>0$ and degree $\delta\geq 2$. If $q > 2(k + 1)\delta^2$, then the following estimate holds:
$$||X(\fq)| - q^k| < 3.2 \delta^{13/3} q^{k- 1/2}.$$
\end{corollary}
\begin{proof}
Suppose $X\subseteq \pmfqb$. We start by assuming that there is no non-zero $h\in S_1(d)$ for which $X\subseteq V(h)$,
because otherwise we can replace $\pmfqb$ by $\PPP^{m-1}(\fqb)$.

Let $S$ be the set of all pairs $(P,\HH)$, where $P\in X(\fq)$ and $\HH$ is a hyperplane of $\pmfq$ that does not contain $P$. We compute the size of $S$ in two ways. We have $|X(\fq)|$ choices for $P$ and once we choose $P$, we have $\pi_m-\pi_{m-1}=q^m$ choices for $\HH$. Therefore, $|S|=q^m |X(\fq)|$.
On the other hand, we also have
$|S|=\sum_{\HH}|(X\setminus \HH) (\fq)|$. There are $\pi_m$ hyperplanes. For each hyperplane $\HH$, we have
\[||(X\setminus \HH) (\fq)| -q^k |
\leq (\delta - 1)(\delta - 2)q^{k- 1/2} + 5\delta^{13/3} q^{k-1}
\leq \delta^2q^{k- 1/2} +5 \frac{\delta^{13/3}}{\sqrt{q}} q^{k-1/2}.\]
Since $q>2(k+1)\delta^2$ and $k\geq 1$, we have $\sqrt{q}\geq 2\delta$. In addition, $\delta\geq 2$ implies $\delta^2<\frac{1}{2}\delta^{10/3}$.
This implies that
\[||(X\setminus \HH) (\fq)| -q^k |<\Big(\frac{1}{2}\delta^{10/3}+\frac{5}{2}\delta^{10/3} \Big) q^{k-1/2} =3\delta^{10/3} q^{k-1/2}.\]
Therefore,
$$||S| - q^k\pi_m|\leq 3\delta^{10/3} q^{k-1/2}\pi_m.$$
We conclude that
\[\Big||X(\fq)|- q^k \frac{\pi_m}{q^m}\Big| =\frac{||S| - q^k\pi_m|}{q^m}
\leq  3\delta^{10/3} q^{k-1/2}\frac{\pi_m}{q^m}.\]
Now
\[\frac{\pi_m}{q^m}<\frac{q^{m+1}}{(q-1)q^m}=1+\frac{1}{q-1}\leq 1+\frac{1}{2(k+1)\delta^2}\leq 1+\frac{1}{16}.\]
The result follows as $3\times 17/16<3.2$.
\end{proof}

Note that Theorem~\ref{thm: affine variety bound} and Corollary~\ref{Cor bound} apply to algebraic sets that are irreducible over $\fq$. Cafure and Matera also have a result for algebraic sets that are irreducible over $\fq$ but not over $\fqb$.
This time, there is no $q^k$ term.

\begin{lemma}\cite[Lemma 2.3]{cafure2006improved}\\
Let $X$ be an affine algebraic set which is defined over $\fq$ and irreducible over $\fq$ but not irreducible over $\fqb$. Denote $\dim(X)=k$ and $\deg(X)=\delta$. Then the following inequality holds
$$|X(\fq)| \leq \frac{\delta^2}{4} q^{k-1}.$$
\end{lemma}

We obtain a projective version of this result by a similar counting argument.
\begin{corollary}\label{cor: bound for non-irreducible}
Let $X$ be an algebraic set which is defined over $\fq$ and irreducible over $\fq$ but not irreducible over $\fqb$. Denote $\dim(X)=k$ and $\deg(X)=\delta$. Then the following inequality holds
$$|X(\fq)| < \frac{\delta^2}{2} q^{k-1}.$$
\end{corollary}
\begin{proof}
Suppose $X\subseteq \pmfqb$. Let $S$ be the set of all pairs $(P,\HH)$, where $P\in X(\fq)$ and $\HH$ is a hyperplane of $\pmfq$ that does not contain $P$. We compute the size of $S$ in two ways. We have $|X(\fq)|$ choices for $P$ and once we choose $P$, we have $\pi_m-\pi_{m-1}=q^m$ choices for $\HH$. Therefore, $|S|=q^m |X(\fq)|$.
On the other hand, we also have
$$|S|=\sum_{\HH}|(X\setminus \HH) (\fq)| \leq \sum_{\HH} \frac{\delta^2}{4}q^{k-1} =\pi_m \frac{\delta^2}{4}q^{k-1}.$$
It follows that
\[|X(\fq)|\leq \frac{\pi_m}{q^m}\frac{\delta^2}{4}q^{k-1} <2\frac{\delta^2}{4}q^{k-1}. \qedhere\]
\end{proof}

Furthermore, we can estimate the dominant term of the formula in Theorem \ref{Thm: complete GDC conjecture}.
\begin{lemma}\label{Lem starts with c q m-l}
Suppose we have $1\leq l\leq m$, $1\leq c\leq d$ and
\begin{align*}
\tbinom{m+d}{d}&-\tbinom{m+d+1-l}{d} <r 
\leq \tbinom{m+d}{d}-\tbinom{m+d+1-l}{d}+\tbinom{m+d+1-l-c}{d-c}.
\end{align*}
Then we have
$$c q^{m-l}< H_{r-\binom{m+d}{d}+\binom{m+d+1-l}{d}}(d-1,m-l+1)+\pi_{m-l}.$$
\end{lemma}
\begin{proof}
See Appendix~\ref{Sec: Appendix}.
\end{proof}

Proposition~\ref{dimension} and Proposition \ref{degree} bound the degree and dimension of $X$. Together with Corollary~\ref{Cor bound} and Corollary~\ref{cor: bound for non-irreducible}, this gives a tight bound on the number of $\fq$-rational points in the highest-dimensional components of $X$. Under our assumption that $X$ does not have a linear $(m-l)$-dimensional subspace, this bound is of order $\frac{c}{2}q^{m-l}$.
Moreover, Proposition~\ref{Prop: Components of codim geq l} gives an upper bound for the number of $\fq$-rational points coming from lower-dimensional components.
Combining this with the upper bound of $\fq$-rational points on $(m-l)$-dimensional components, we show that for sufficiently large $q$, the number of $\fq$-rational points in $X$ is smaller than $c q^{m-l}$. Then, we are done by Lemma~\ref{Lem starts with c q m-l}.

\begin{proposition}\label{X does not have linear subspace}
Suppose we have $1\leq l\leq m$, $1\leq c\leq d$ and
\begin{align*}
\tbinom{m+d}{d}&-\tbinom{m+d+1-l}{d}+\tbinom{m+d-l-c}{d-c-1} <r 
\leq \tbinom{m+d}{d}-\tbinom{m+d+1-l}{d}+\tbinom{m+d+1-l-c}{d-c}.
\end{align*}
Suppose $F_1,\dots,F_r$ are linearly independent polynomials in $S_d(m)$ and $X=V(F_1,\dots,F_r)$ does not contain a $(m-l)$-dimensional linear subspace defined over $\fq$. If
$$q\geq \max\Big\{2(m-l+1)c^2+1,\;8\frac{d^{l+1}}{c},\; 164 c^{14/3}\Big\},$$
then we have
$$|X(\fq)|< H_{r-\binom{m+d}{d}+\binom{m+d+1-l}{d}}(d-1,m-l+1)+\pi_{m-l}.$$    
\end{proposition}
\begin{proof}
We know from Proposition \ref{dimension} and Proposition~\ref{degree} that $\dim(X)\leq m-l$ and $\deg_{m-l}(X)\leq c$.
Let $X_1$ be the union of the irreducible components of $X$ of dimension $m-l$ and $X_2$ be the union of the irreducible components of dimension at most $m-l-1$.
By Proposition \ref{Prop: Components of codim geq l}, we know that
$$|X_2(\fq)|\leq d^{l+1}\pi_{m-l-1}< d^{l+1}q^{m-l-1}\frac{q}{q-1}\leq 2 d^{l+1}q^{m-l-1}.$$
Write $X_1$ as the union components irreducible over $\fq$. Suppose $Y_1,\dots,Y_{s}$ are the components that are also irreducible over $\fqb$ and $Z_1,\dots,Z_t$ are the components that are not irreducible over $\fqb$ (but are irreducible over $\fq$). Denote $\deg(Y_i)=d_i$ and $\deg(Z_i)=e_i$. Since $X$ does not contain a $m-l$ dimensional linear subspace defined over $\fq$, we know that all $d_i\geq 2$. 
Since $q>2(m-l+1)c^2$, by Corollary \ref{Cor bound} we know that
\begin{align*}
|Y_i(\fq)| &< q^{m-l}+ 3.2 d_i^{10/3} q^{m-l-1/2}
\leq \frac{d_i}{2} q^{m-l}+ 3.2 d_i^{10/3} q^{m-l- 1/2}  .
\end{align*}
Next, we know that
$$|Z_i(\fq)|\leq \frac{e_i^2}{2} q^{m-l-1}<\frac{e_i}{2} q^{m-l}.$$
Therefore, it follows that
\begin{align*}
|X_1(\fq)|
&< \Big(\sum d_i/2 +\sum e_i/2\Big) q^{m-l} + 3.2 \Big(\sum d_i\Big)^{10/3} q^{m-l-1/2}\\
&\leq \frac{c}{2} q^{m-l} +3.2 c^{10/3} q^{m-l-1/2}.
\end{align*}
We see that
\begin{align*}
|X(\fq)|
&< \frac{c}{2} q^{m-l} +3.2 c^{10/3} q^{m-l-1/2} +2d^{l+1} q^{m-l-1}.
\end{align*}
Since $q\geq 164 c^{14/3}$, we have $3.2 c^{10/3} q^{m-l-1/2} <\frac{c}{4} q^{m-l}$. Moreover, since $q\geq 8 \frac{d^{l+1}}{c}$, we have $2d^{l+1} q^{m-l-1}\leq \frac{c}{4} q^{m-l}$. We conclude that $|X(\fq)|< c q^{m-l}$.
We are done by Lemma~\ref{Lem starts with c q m-l}.
\end{proof}

\section{If $X$ contains a linear subspace}\label{sec: if X contains a linear subspace}
In this section, we focus on proving the case of Theorem~\ref{Thm: complete GDC conjecture} when $X$ contains a linear subspace of dimension $(m-l)$ defined over $\fq$. In this case, equality can actually hold, meaning the number of $\fq$-rational points in $X$ could be equal to the conjectured formula.
Therefore, we cannot rely on the bounds provided by Corollary~\ref{Cor bound}, Corollary~\ref{cor: bound for non-irreducible} and Proposition~\ref{Prop: Components of codim geq l}.
However, we leverage the existence of a $(m-l)$-dimensional linear subspace, to divide $X$ into parts whose number of $\fq$-rational points can be bounded with the help of Theorem~\ref{Affine result}.

\begin{lemma}\label{Lem: points less than sum of Hrk}
Suppose we have $1\leq l\leq m$ and
\begin{align*}
\tbinom{m+d}{d}-\tbinom{m+d+1-l}{d} <r \leq \tbinom{m+d}{d}-\tbinom{m+d-l}{d}.
\end{align*}
Suppose $F_1,\dots,F_r$ are linearly independent polynomials in $S_d(m)$ and $X=V(F_1,\dots,F_r)$ contains a $(m-l)$-dimensional linear subspace defined over $\fq$. If $q>d$, then there are $r_1,\dots,r_{l}$ such that $r=\sum r_i$ and $0\leq r_i\leq \binom{m+d-i}{d-1}$ such that
$$|X(\fq)| \leq \sum_{i=1}^{l}H_{r_i}(d-1,m+1-i)+\pi_{m-l}.$$
\end{lemma}
\begin{proof}
After a linear change of variables, we can assume that $V(X_0,\dots,X_{l-1})\subseteq X$.
Since the original $m-l$ dimensional linear subspace was defined over $\fq$, this change of variables does not change the number of $\fq$-rational points.
Therefore, $F_1,\dots, F_r$ are all inside the ideal generated by $X_0,\dots, X_{l-1}$. Let $W$ be the vector space generated by $F_1,\dots, F_r$ and for $1\leq i\leq l$, let $W_i$ be the intersection of $W$ with the ideal generated by $X_0,\dots,X_{i-1}$. Set $W_0=\{0\}$. So $W_0\subseteq W_1\subseteq \dots \subseteq W_l=W$. Let $r_i=\dim(W_i)-\dim(W_{i-1})$. So we have $\sum r_i =\dim(W)=r$ and
\begin{align*}
0&\leq r_i =\dim(W_i/W_{i-1}) \leq \dim\Big((S_d(m)\cap \langle X_0,\dots,X_{i-1}\rangle) \Big/  (S_d(m)\cap \langle X_0,\dots,X_{i-2}\rangle)\Big)\\
&= \dim(X_{i-1} \fq[X_{i-1},\dots,X_m]_{d-1} )=\tbinom{m+d-i}{d-1}.
\end{align*}
Now we have
$$X=V(X_0,\dots X_{l-1}) \cup \bigcup_{i=1}^{l} X\cap \big( V(X_0,\dots, X_{i-2}) \setminus V(X_{i-1}) \big).$$
We have $|V(X_0,\dots X_{l-1}) (\fq)|=\pi_{m-l}$. For $1\leq i\leq l$, let $G_1,\dots, G_{r_i}$ be polynomials in $W_i$ that form a basis of $W_i/ W_{i-1}$. So
$$X\cap \big( V(X_0,\dots, X_{i-2})\setminus V(X_{i-1})\big) \subseteq V(X_0,\dots, X_{i-2}, G_1,\dots, G_{r_i}) \setminus V(X_{i-1}).$$
Let $g_j$ be the polynomial obtained from $G_j$ by plugging in $X_0=\dots,X_{i-2}=0$ and $X_{i-1}=1$. So $g_1,\dots, g_{r_{i}}$ are linearly independent polynomials in $\fq[X_{i},\dots, X_{m}]$ with degree at most $d-1$. Moreover, the number of points of $V(X_0,\dots, X_{i-2}, G_1,\dots, G_{r_i}) \setminus V(X_{i-1})$ in $\pmfq$ is the same as the number of points of $Z(g_1,\dots,g_{r_i})$ in $\mathbb{A}^{m+1-i}(\fq)$. By Theorem \ref{Affine result}, this is at most $H_{r_i}(d-1,m+1-i)$. The result follows.
\end{proof}

The following is a technical lemma that involves sums of $H_r(d,m)$ terms. Its proof is deferred to Appendix~\ref{Sec: Appendix}.

\begin{lemma}\label{Lem: sum of Hrk}
Suppose we are given $2 \leq l \leq m$ and $r_1, \dots, r_l$ such that $0\leq r_k\leq \binom{m+d-k}{d-1}$ and
\begin{align*}
\tbinom{m+d}{d}-\tbinom{m+d+1-l}{d} <\sum_{k=1}^l r_k \leq \tbinom{m+d}{d}-\tbinom{m+d-l}{d}.
\end{align*}
Let $r=\sum_{k=1}^l r_k$ and $r'=r-\binom{m+d}{d}+\binom{m+d+1-l}{d}$. If $q\geq d$, then we have
\[\sum_{k=1}^l H_{r_k}(d-1,m-k+1) \leq H_{r'} (d-1,m-l+1).\]    
\end{lemma}
\begin{proof}
See Appendix~\ref{Sec: Appendix}.
\end{proof}

\begin{proposition}\label{X Contains a linear subspace}
Suppose we have $1\leq l\leq m$, $1\leq c\leq d$ and
\begin{align*}
\tbinom{m+d}{d}-\tbinom{m+d+1-l}{d} <r \leq \tbinom{m+d}{d}-\tbinom{m+d-l}{d}.
\end{align*}
Suppose $F_1,\dots,F_r$ are linearly independent polynomials in $S_d(m)$ and $X=V(F_1,\dots,F_r)$ contains a $(m-l)$-dimensional linear subspace defined over $\fq$. If $q>d$, then we have
$$|X(\fq)|\leq H_{r-\binom{m+d}{d}+\binom{m+d+1-l}{d}}(d-1,m-l+1)+\pi_{m-l}.$$    
\end{proposition}
\begin{proof}
This follows from Lemma \ref{Lem: points less than sum of Hrk} and Lemma \ref{Lem: sum of Hrk}.
\end{proof}

\begin{proof}[Proof of Theorem \ref{Thm: complete GDC conjecture}]
Proposition~\ref{Lower bound for erdm} shows that the conjectured formula is a lower bound for $e_r(d,m)$. Proposition~\ref{X does not have linear subspace} and Proposition~\ref{X Contains a linear subspace} together show that the conjectured formula is an upper bound for $e_r(d,m)$.
\end{proof}

\section{Case $r\leq {m+d-1 \choose d-1}$}\label{sec: proof for incomplete gdc}

In this section, our goal is to prove Theorem \ref{Thm l=1 with e}. We will do this by induction on $m$. The arguments of this section involve several technical lemmas that will be proven in the Appendix \ref{Sec: Appendix}.

Given a linear subspace $W\subseteq S_d(m)$ with $\dim(W)=r$ and $L\in S_1(m)$, Beelen, Datta and Ghorpade in \cite{beelen2018maximum} define $t_W(L)=\dim(W\cap LS_{d-1}(m))$ and
\[t_W=\max\{t_W(L)\mid L\in S_1(m)\}.\]
Note that $0\leq t_W\leq r$. We will prove Theorem~\ref{Thm l=1 with e} by induction on $m$. We divide the induction step into several lemmas depending on which range $t_W$ is in.

\begin{lemma}
\cite[Lemma 2.9]{beelen2018maximum}\label{linear factor}
Suppose $q \geq d$ and $F_1,\dots,F_r$ are linearly independent in $S_d(m)$. If there is some $L \in S_1(m)$ that divides $F_1, \dots, F_r$, then we have
$$|V(F_1, \dots, F_r)(\fq)| \leq H_r(d-1,m)+\pi_{m-1}.$$
\end{lemma}
The Lemma~\ref{linear factor} states that if $t_W=r$, then $|V(W)(\fq)|\leq H_r(d-1,m)+\pi_{m-1}$.
Next, we will consider the cases:
\begin{enumerate}
    \item $t_W \leq \binom{m+e-1}{e-1}$;
    \item $\binom{m+e}{e}< t_W <r$;
    \item $\binom{m+e-1}{e-1}< t_W \leq\binom{m+e}{e}$.
\end{enumerate}

We start with the case $t_W \leq \binom{m+e-1}{e-1}$, we will use the following lemma.
\begin{lemma}\cite[Lemma 2.5]{beelen2018maximum}\label{any hyperplane}
Let $X$ be any subset of $\mathbb{P}^m(\fq)$. Let $a=\max_{\HH}|X \cap \HH|$, where the max is taken over all hyperplanes. Then $|X| \leq aq+1$. If $X \neq \PPP^m(\fq)$, then $|X| \leq aq$.
\end{lemma}

\begin{lemma}\label{small t}
Suppose $m\geq 2$, $d\geq 2$, $0\leq e\leq d-2$ and $\binom{m+e}{e}< r \leq \binom{m+e+1}{e+1}$.
Let $W \subseteq S_d(m)$ be a linear subspace of dimension $r$ with $t_W \leq \binom{m+e-1}{e-1}$. If for each $0\leq s\leq \binom{m-1+e+1}{e+1}$ we have $e_{s}(d,m-1)\leq H_{s}(d-1,m-1)+\pi_{m-2}$, then we also have $|V(W)(\fq)| < H_r(d-1,m)+\pi_{m-1}$.
\end{lemma}
\begin{proof}
Consider a hyperplane $\HH$ given by $L=0$ for some non-zero $L\in S_1(m)$. Let $L,L_0,\dots,L_{m-1}$ be a basis of $S_1(m)$. Consider the map $\phi_{\HH}: S_d(m) \to S_{d}(m-1)$ in which a homogeneous polynomial in $\fq[x_0,\dots,x_m]$ is written in terms of $L,L_0,\dots,L_{m-1}$ and then we plug in $L=0$, $L_0=y_0,\dots,L_{m-1}=y_{m-1}$. Note that $\phi_{\HH}$ induces an isomorphism $S_d(m)/(LS_{d-1}(m)) \cong S_{d}(m-1)$. 
 
Denote $t=t_W(L)$, so $t=\dim(W\cap \ker(\phi_{\HH}))$ and $r-t=\dim(\phi_{\HH}(W))$. Note that
\[|V(W)(\fq)\cap \HH|
=|V( W)(\fq)\cap V(L)|
=|V(\phi_{\HH}(W))(\fq)| \leq  e_{r-t}(d,m-1).\]
By Lemma~\ref{any hyperplane}, we have
$|V(W)(\fq)| \leq q e_{r-t}(d,m-1)$.

\begin{itemize}
    \item Case 1: $r-t\leq \binom{m-1+e+1}{e+1}$, then we have $e_{r-t}(d,m-1)\leq H_{r-t}(d-1,m-1)+\pi_{m-2}$.
    Recall that  $t \leq \binom{m+e-1}{e-1}$. By Lemma \ref{Hrdm Hrdm-1}, we have
    \[q(H_{r-t}(d-1,m-1)+\pi_{m-2})
    \leq H_r(d-1,m)+q\pi_{m-2} = H_r(d-1,m)+\pi_{m-1}-1.\]
    \item Case 2: $\binom{m-1+e+1}{e+1}\leq r-t$. Then, by Lemma~\ref{Lem: Special value of Hrdm} we have
    \begin{align*}
    e_{r-t}(d,m-1)
    &\leq e_{\binom{m-1+e+1}{e+1}}(d,m-1)
    \leq H_{\binom{m-1+e+1}{e+1}}(d-1,m-1)+\pi_{m-2}\\
    &=(d-1-(e+1))q^{m-2}+\pi_{m-2}.  
    \end{align*}
    By Lemma~\ref{Lem: Special value of Hrdm}, we see that
    \begin{align*}
    &q((d-e-2)q^{m-2}+\pi_{m-2}) \\
    &= H_{\binom{m+e+1}{e+1}}(d-1,m) +\pi_{m-1}-1
    < H_{r}(d-1,m)+\pi_{m-1}. \qedhere
    \end{align*}
\end{itemize}
\end{proof}

Next, we consider the case where $\binom{m+e}{e}< t_W <r$. We will make use of the following lemma.
\begin{lemma}\label{rec2}
Assume that $1<r \leq \binom{m+d-1}{d-1}$.
Then for any $r$ dimensional subspace $W$ of $S_d(m)$ with $t_W=t$, we have 
$$|V(W)(\fq)| \leq e_{r-t}(d,m-1)+H_t(d-1,m).$$
Moreover if $t\geq 2$ and $\gcd(W)=1$, then we have
$$|V(W)(\fq)|\leq e_{r-t}(d,m-1) +(d-1)^2 q^{m-2}.$$
\end{lemma}
\begin{proof}
The proof is very similar to the proof \cite[Lemma 4.1]{beelen2018maximum}, so we omit it here.
\end{proof}

\begin{lemma}\label{large t}
Suppose $m\geq 2$, $d\geq 2$, $0\leq e\leq d-2$ and $\binom{m+e}{e}< r \leq \binom{m+e+1}{e+1}$.
Let $W \subseteq S_d(m)$ be a linear subspace of dimension $r$ with $\binom{m+e}{e}< t_W <r$. If $q\geq d+e$, then $|V(W)(\fq)| \leq H_r(d-1,m)+\pi_{m-1}$.
\end{lemma}
\begin{proof}
Let $w_{r}(d-1,m)=(\beta_1,\dots,\beta_{m+1})$.
By Lemma~\ref{bound on betaone}, $\binom{m+e}{e}< r \leq \binom{m+e+1}{e+1}$ implies that $\beta_1=d-2-e$. By Lemma \ref{rec2}, we have
$$|V(W)(\fq)| \leq e_{r-t_W}(d,m-1)+H_{t_W}(d-1,m).$$
Since $r-t_W \geq 1$ and $t_W > \binom{m+e}{e}$, we have
\begin{align*}
    |V(W)(\fq)| &\leq e_{1}(d,m-1)+H_{\binom{m+e}{e}+1}(d-1,m)\\
    &= (d-1)q^{m-2}+\pi_{m-2}+(d-2-e)q^{m-1}+q^{m-2}(e+1) \\
    &=\beta_1q^{m-1}+(d+e)q^{m-2}+\pi_{m-2}
     \leq \beta_1q^{m-1}+\pi_{m-1} \\
    & \leq H_r(d-1,m)+\pi_{m-1}.\qedhere
\end{align*}
\end{proof}

Finally, we consider the case $\binom{m+e-1}{e-1}< t_W \leq\binom{m+e}{e}$. We will apply the following proposition of Homma and Kim.
\begin{proposition}
\cite[Theorem 1.2]{homma2013elementary}\label{no linear divisor hypersurface}
For $d\geq 2$, consider $G\in S_d(m)$, $G\neq 0$. Assume that $G$ has no linear factors. Then, we have
$$|V(G)(\fq)| \leq (d-1)q^{m-1}+dq^{m-2}+\pi_{m-3}.$$
\end{proposition}

\begin{lemma}\label{medium t}
Suppose $m\geq 2$, $d\geq 2$, $0\leq e\leq d-2$ and $\binom{m+e}{e}< r \leq \binom{m+e+1}{e+1}$.
Let $W \subseteq S_d(m)$ be a linear subspace of dimension $r$ with $\binom{m+e-1}{e-1}< t_W \leq\binom{m+e}{e}$. Suppose that
\[q\geq \max\{d+e+\frac{e^2-1}{d-(e+1)}, d-1+e^2-e\}\]
and for each $s,d_1$ satisfying $0\leq s\leq \binom{m-1+e+1}{e+1}$ and $1\leq d_1\leq d$, we have
\[e_{s}(d_1,m-1)\leq H_{s}(d_1-1,m-1)+\pi_{m-2}.\] Then we have \[|V(W)(\fq)| \leq H_r(d-1,m)+\pi_{m-1}.\]
\end{lemma}
\begin{proof}
Denote the g.c.d. of all polynomials in $W$ as $G=\gcd(W)$. Let $c_1=\deg(G)$ and $t=t_W$. Suppose $W=GW_1$ with $W_1\subseteq S_{d-c_1}(m)$ and $\gcd(W_1)=1$. Since $t\neq r$, we know that $G$ has no linear factors and $t=t_W=t_{W_1}$.
Since $t\leq \binom{m+e}{e}$ and $\gcd(W_1)=1$, by Lemma \ref{rec2}  we know that
\begin{align*}
|V(W_1)(\fq)|&\leq e_{r-t}(d-c_1,m-1) +(d-c_1-1)^2 q^{m-2}\\
&\leq e_{r-\binom{m+e}{e}}(d-c_1,m-1) +(d-c_1-1)^2 q^{m-2}.
\end{align*}
Now $r-\binom{m+e}{e}\leq \binom{m+e+1}{e+1}-\binom{m+e}{e}=\binom{m-1+e+1}{e+1}$, so
\[e_{r-\binom{m+e}{e}}(d-c_1,m-1)\leq H_{r-\binom{m+e}{e}}(d-c_1-1,m-1)+\pi_{m-2}.\]
By Lemma \ref{relation Hr} and Lemma \ref{Lem d-c to d} this implies that
\begin{align*}
&|V(W_1)(\fq)|\leq H_{r-\binom{m+e}{e}}(d-c_1-1,m-1)+\pi_{m-2} +(d-c_1-1)^2 q^{m-2}\\
&= H_{r}(d-c_1-1,m)-(d-c_1-2-e)q^{m-2}(q-1)\\
&+\pi_{m-1}-q^{m-1} +(d-c_1-1)^2 q^{m-2}\\
&=H_{r}(d-c_1-1,m)+c_1q^{m-1}+ \pi_{m-1}-(d-1-e)q^{m-1}\\
&+((d-c_1-1)^2+d-c_1-2-e)q^{m-2}\\
&\leq H_{r}(d-1,m)+ \pi_{m-1}-(d-1-e)q^{m-1}
+((d-c_1-1)^2+d-c_1-2-e)q^{m-2}.
\end{align*}
Now if $c_1=0$, then $W=W_1$. Moreover, $d+e+\frac{e^2-1}{d-(e+1)}\leq q$ implies that
$$\big((d-1)^2 +d-2-e\big)q^{m-2}\leq (d-1-e)q^{m-1}.$$
This shows that $|V(W)(\fq)|\leq H_{r}(d-1,m)+ \pi_{m-1}$.

On the other hand, if $c_1\neq 0$, then Proposition~\ref{no linear divisor hypersurface} implies that
$$|V(G)(\fq)| \leq (c_1-1)q^{m-1}+c_1q^{m-2}+\pi_{m-3}.$$
Then we have
\begin{align*}
|V(W)(\fq)| &\leq |V(W_1)(\fq)|+ |V(G)(\fq)|\\
&\leq H_{r}(d-1,m)+ \pi_{m-1}-(d-c_1-e)q^{m-1}\\
&+((d-c_1-1)^2+d-2-e)q^{m-2} +\pi_{m-3}.
\end{align*}
Therefore, we will be done if we show that
$$\big((d-c_1-1)^2 +d-2-e\big)q^{m-2}+\pi_{m-3}\leq (d-c_1-e)q^{m-1}.$$
For this it is enough to check that $\frac{(d-c_1-1)^2 +d-2-e}{d-c_1-e}\leq q-1$. Since $\dim(W_1)=r$ and $W_1\subseteq S_{d-c_1}(m)$, we know that
$\tbinom{m+e}{e} < r\leq \tbinom{m+d-c_1}{m},$
that is $e+1\leq d-c_1$. Consider the function $f(x)=\frac{(x-1)^2+d-2-e}{x-e}$ on the interval $e+1\leq x\leq d$. It is easy to see from elementary calculus that $f$ must be maximized at an end point. Now $f(e)=d-2+e^2-e\leq q-1$ and $$f(d)=d+e-1+\frac{e^2-2e-1}{d-e}\leq d+e-1+\frac{e^2-1}{d-(e+1)}\leq q-1.$$
The result follows.
\end{proof}

\begin{customthm}{\ref{Thm l=1 with e}}
Suppose $m\geq 2$, $d\geq 2$, $0\leq e\leq d-2$ and $\binom{m+e}{e}< r \leq \binom{m+e+1}{e+1}$.
If $q\geq \max\{d+e+\frac{e^2-1}{d-(e+1)}, d-1+e^2-e\}$, then we have
$$e_r(d,m)= H_r(d-1,m)+\pi_{m-1}.$$
\end{customthm}
\begin{proof}
The case $d=1$ was proved in \cite{datta2017number}, so assume $d\geq 2$.
We induct on $m$. The base case $m=1$ is shown in \cite{datta2017number}. The induction step follows from Lemma \ref{linear factor}, Lemma \ref{small t}, Lemma \ref{large t} and Lemma \ref{medium t}.
\end{proof}

\appendix

\section{Technical Lemmas}\label{Sec: Appendix}

\subsection{Lemmas for Section \ref{sec: case when X does not contain a linear subspace}}

The following lemma is a restatement of \cite[Theorem 3.1]{beelen2019note}.
\begin{lemma}\label{Lem: r in term of Hrdm}
Suppose $H_r(d,m)=\sum_{j=1}^{d}\lfloor q^{a_j}\rfloor$ for some $-1\leq a_1\leq a_2\dots\leq a_d\leq m-1$ and $1\leq r\leq \binom{m+d}{d}$. Then (for $q>d$) we have
$$r=\tbinom{m+d}{d}-\sum_{j=1}^{d}\tbinom{a_j+j}{j}.$$    
\end{lemma}
\begin{proof}
See \cite[Theorem 3.1]{beelen2019note}.
\end{proof}

We use this to derive a lemma about special values of $H_r(d,m)$.

\begin{lemma}\label{Lem: Special value of Hrdm}
For $0\leq e\leq d$, we have $H_{\binom{m+e}{e}}(d,m)=(d-e)q^{m-1}$.
\end{lemma}
\begin{proof}
By Lemma \ref{Lem: r in term of Hrdm}, $H_r(d,m)=(d-e)q^{m-1}$ for
\[r =\tbinom{m+d}{d}-\sum_{j=1}^{e} \tbinom{-1+j}{j} -\sum_{j=e+1}^{d} \tbinom{m-1+j}{j} =\tbinom{m+e}{e}.\qedhere\]
\end{proof}

\begin{proof}[Proof of Lemma \ref{Lem starts with c q m-l}]
Let $r_1=r-\binom{m+d}{d}+\binom{m+d+1-l}{d}$, so
$0<r_1\leq \tbinom{m-l+1+d-c}{d-c}$.
Lemma~\ref{Lem: Special value of Hrdm} says that
\[H_{\binom{m-l+1+d-c}{d-c}}(d-1,m-l+1)=((d-1)-(d-c))q^{m-l+1-1}=(c-1)q^{m-l}.\]
This implies that $(c-1)q^{m-l}\leq H_{r_1}(d-1,m-l+1)$. Since $q^{m-l}<\pi_{m-l}$, we conclude that $c q^{m-l}< H_{r_1}(d-1,m-l+1)+\pi_{m-l}$.
\end{proof}

\subsection{Lemmas for Section \ref{sec: if X contains a linear subspace}}


\begin{lemma}\label{r as function of alpha}
If $w_r(d,m)=(a_1, \dots, a_{m+1})$, then
\begin{align*}
    r&=1+\sum_{k=1}^m \tbinom{m-k+d-\sum_{j=1}^k a_j}{m-k+1}.
\end{align*}
\end{lemma}
\begin{proof}
Let $S$ denote $\{w \in \Omega(d,m): w >_{\text{lex}} w_r(d,m)\}$. We know $r=1+|S|$. Note that $S=\bigsqcup_{k=1}^{m} S_k$, where
$$S_k=\{(b_1, \dots, b_{m+1}) \in \Omega(d,m): b_i=a_i \text{ for } i \leq k-1, b_k \geq a_{k}+1\}.$$
Now,
\[|S_k|=|\Omega(d-1-\sum_{j=1}^k a_k, m-(k-1))|=\tbinom{m-k+d-\sum_{j=1}^k a_j}{m-k+1}.\] Therefore, 
\[    r=1+\Big|\bigsqcup_{k=1}^m S_k\Big| =1+\sum_{k=1}^m \tbinom{m-k+d-\sum_{j=1}^ka_j}{m-k+1}. \qedhere\]
\end{proof}

\begin{lemma}
For any non-negative integers $a, b, m, n$ with $n\geq 1$, we have
\[\tbinom{m-a}{n}+ \tbinom{m-b}{n} \leq \tbinom{m-a-b}{n}+ \tbinom{m}{n}.\]
\end{lemma}
\begin{proof}
We have the identity $\binom{n}{k}=\sum_{s=k-1}^{n-1} \binom{s}{k-1}$. Therefore,
\begin{align*}
    \tbinom{m}{n}-\tbinom{m-b}{n} &=\sum_{s=n-1}^{m-1}\tbinom{s}{n-1}-\sum_{s=n-1}^{m-b-1}\tbinom{s}{n-1} 
    =\sum_{s=m-b}^{m-1}\tbinom{s}{n-1}=\sum_{t=0}^{b-1}\tbinom{t+m-b}{n-1}.\\
\end{align*}
Similarly, $\binom{m-a}{n}-\binom{m-a-b}{n}=\sum_{t=0}^{b-1}\binom{t+m-a-b}{n-1}$. It follows that
\begin{align*}
\tbinom{m-a}{n}-\tbinom{m-a-b}{n}
&=\sum_{t=0}^{b-1}\tbinom{t+m-a-b}{n-1} 
\leq \sum_{t=0}^{b-1}\tbinom{t+m-b}{n-1} =\tbinom{m}{n}-\tbinom{m-b}{n}.\qedhere  
\end{align*}
\end{proof}

\begin{lemma}\label{Lem: Sum of Hsk(d,m)}
Suppose that we have $1\leq s_1,\dots,s_l\leq \binom{m+d}{d}$ with $\sum_{k=1}^l s_k> (l-1)\binom{m+d}{d}$. Let $r=\sum_{k=1}^l s_k -(l-1)\binom{m+d}{d}$. If $q\geq d+1$, then we have
$$\sum_{k=1}^{l}H_{s_k}(d,m)\leq H_{r}(d,m).$$    
\end{lemma}
\begin{proof}
First we claim that for a given $d$, if we prove the result for $l=2$, then we will automatically have it for all $l$ (and the given $d$). We show this by induction on $l$. Suppose that the result is known for $l-1$. Since $s_l\leq \binom{m+d}{d}$, we have
$$\sum_{i=1}^{l-1}s_i > (l-1)\tbinom{m+d}{d}-\tbinom{m+d}{d} =(l-2)\tbinom{m+d}{d}.$$
Let $r_1=\sum_{k=1}^{l-1}s_k -(l-2)\binom{m+d}{d}$ and $r=\sum_{k=1}^{l}s_k -(l-1)\binom{m+d}{d}$. By induction hypothesis we have $\sum_{k=1}^{l-1}H_{s_k}(d,m)\leq H_{r_1}(d,m)$. Finally $r_1+s_l= r+\binom{m+d}{d}>\binom{m+d}{d}$, so by using the result for $l=2$ we see that
$$\sum_{k=1}^{l}H_{s_k}(d,m)\leq H_{r_1}(d,m)+H_{s_k}(d,m)\leq H_{r}(d,m).$$
Thus, if we prove the result for a given $d$ with $l=2$, then we prove it for all $l$ with that given $d$.

Now, consider the case $l=2$. Write $H_{s_1}(d,m)=\sum_{j=1}^{d}\lfloor q^{a_j}\rfloor$, $H_{s_2}(d,m)=\sum_{j=1}^{d}\lfloor q^{b_j}\rfloor$ and $H_{r}(d,m)=\sum_{j=1}^{d}\lfloor q^{c_j}\rfloor$ with $-1\leq a_1\leq \dots\leq a_d\leq m-1$, $-1\leq b_1\leq \dots\leq b_d\leq m-1$, $-1\leq c_1\leq \dots\leq c_d\leq m-1$ and $r=s_1+s_2-\binom{m+d}{d}$. So by Lemma \ref{Lem: r in term of Hrdm} we have
$s_1=\binom{m+d}{d}-\sum_{j=1}^{d}\binom{a_j+j}{j}$, $s_2=\binom{m+d}{d}-\sum_{j=1}^{d}\binom{b_j+j}{j}$ and $r=\binom{m+d}{d}-\sum_{j=1}^{d}\binom{c_j+j}{j}$.
Thus
$$\sum_{j=1}^{d}\tbinom{a_j+j}{j}+ \sum_{j=1}^{d}\tbinom{b_j+j}{j}=\sum_{j=1}^{d}\tbinom{c_j+j}{j}.$$
We want to show that
$$\sum_{j=1}^{d}\lfloor q^{a_j}\rfloor +\sum_{j=1}^{d}\lfloor q^{b_j}\rfloor \leq \sum_{j=1}^{d}\lfloor q^{c_j}\rfloor.$$

We will prove the result by induction on $d$. The base case is $d=1$. This means $(a_1+1)+(b_1+1)=c_1+1$. Since $q^{a_1}+q^{b_1}\leq q^{a_1+b_1+1}$, the result holds for $d=1$ with any $m$.

Next, for the induction step, consider some $d\geq 2$ and assume that the result has been shown for smaller values of $d$ with any $m$. Without loss of generality assume that $b_d\leq a_d$ and denote $c_d=c$.
\begin{itemize}
    \item Case 1: $a_d\geq c+1$. Then we have
    \begin{align*}
    \tbinom{c+1+d}{d}&\leq \tbinom{a_d+d}{d}\leq \sum_{j=1}^{d}\tbinom{a_j+j}{j}+ \sum_{j=1}^{d}\tbinom{b_j+j}{j}
    =\sum_{j=1}^{d}\tbinom{c_j+j}{j}
    \leq \sum_{j=1}^{d}\tbinom{c+j}{j} =\tbinom{c+d+1}{d}-1.
    \end{align*}
    This is a contradiction, therefore $a_d\leq c$.
    \item Case 2: $a_d=c$. Then
    \begin{align*}
    \sum_{j=1}^{d-1}\tbinom{c_j+j}{j}
    &=\tbinom{b_d+d}{d}+ \sum_{j=1}^{d-1}\tbinom{a_j+j}{j}+ \sum_{j=1}^{d-1}\tbinom{b_j+j}{j}
    \geq \tbinom{b_d+d-1}{d-1}+ \sum_{j=1}^{d-1}\tbinom{a_j+j}{j}+ \sum_{j=1}^{d-1}\tbinom{b_j+j}{j}.
    \end{align*}
    Choose a sufficiently large $m_1$ such that $\sum_{j=1}^{d-1}\binom{c_j+j}{j}< \binom{m_1+d-1}{d-1}$.
    Let $s_1'=\binom{m_1+d-1}{d-1}-\sum_{j=1}^{d-1}\binom{a_j+j}{j}$, $s_2'=\binom{m_1+d-1}{d-1}-\sum_{j=1}^{d-1}\binom{b_j+j}{j}$, $s_3'=\binom{m_1+d-1}{d-1}-\binom{b_d+d-1}{d-1}$.
    So we have $1\leq s_1',s_2',s_3'\leq \binom{m_1+d-1}{d-1}$ and
    \begin{align*}
    s_1'+s_2'+s_3'&= 3\tbinom{m_1+d-1}{d-1}- \sum_{j=1}^{d-1}\tbinom{a_j+j}{j} -\sum_{j=1}^{d-1}\tbinom{b_j+j}{j} -\tbinom{b_d+d-1}{d-1}\\
    &\geq 3\tbinom{m_1+d-1}{d-1}- \sum_{j=1}^{d-1}\tbinom{c_j+j}{j}
    >2 \tbinom{m_1+d-1}{d-1}.
    \end{align*}
    Let $r_1=s_1'+s_2'+s_3'-2\binom{m_1+d-1}{d-1}$ and $r_2= \binom{m_1+d-1}{d-1}-\sum_{j=1}^{d-1}\binom{c_j+j}{j}$, so $r_1\geq r_2>0$. Therefore, by the inductive hypothesis hypothesis we have
    $$H_{s_1'}(d-1,m_1)+H_{s_2'}(d-1,m_1)+H_{s_3'}(d-1,m_1)\leq H_{r_1}(d-1,m_1)\leq H_{r_2}(d-1,m_1).$$
    Next, by Lemma \ref{Lem: r in term of Hrdm} we have
    \begin{align*}
    H_{s_1'}(d-1,m_1) &=\sum_{j=1}^{d-1}\lfloor q^{a_j}\rfloor =H_{s_1}(d,m)- \lfloor q^c\rfloor,\\
    H_{s_2'}(d-1,m_1) &=\sum_{j=1}^{d-1}\lfloor q^{b_j}\rfloor =H_{s_2}(d,m)- \lfloor q^{b_d}\rfloor,\\
    H_{s_3'}(d-1,m_1) &= \lfloor q^{b_d}\rfloor,\\
    H_{r_2}(d-1,m_1) &=\sum_{j=1}^{d-1}\lfloor q^{c_j}\rfloor =H_{r}(d,m)- \lfloor q^c\rfloor.
    \end{align*}
    Therefore, we conclude that $H_{s_1}(d,m)+H_{s_2}(d,m)\leq H_{r}(d,m)$ as required.
    \item Case 3: $a_d\leq c-1$ and $c=0$. Then all $a_i,b_i$ are $-1$, which implies $H_{s_1}(d,m)= H_{s_2}(d,m)=0$, so we are done.
    \item Case 4: $a_d\leq c-1$, $c\geq 1$ and $c_{d-1}=c$. We have
    \begin{align*}
     H_{s_1}(d,m)+ H_{s_2}(d,m)
     &= \sum_{j=1}^{d}\lfloor q^{a_j}\rfloor +\sum_{j=1}^{d}\lfloor q^{b_j}\rfloor
     \leq 2d q^{c-1} <2q^c\leq H_{r}(d,m).
    \end{align*}
    \item Case 5: $a_d\leq c-1$, $c\geq 1$ and $c_{d-1}\leq c-1$. Suppose $c-1$ occurs $k_1$ times in the list $a_1,\dots,a_d$, $k_2$ times in $b_1,\dots,b_d$ and $k_3$ times in $c_1,\dots,c_{d-1}$. Therefore, $H_{s_1}(d,m)+H_{s_2}(d,m)\leq (k_1+k_2) q^{c-1}+ (2d-k_1-k_2)q^{c-2}$ and $H_r(d,m)\geq q^c+ k_3 q^{c-1}$. So it is enough to show that
    $$(k_1+k_2-k_3)+\frac{2d-k_1-k_2}{q}\leq q.$$
    Case 5A: $k_1+k_2\leq d-2$. Then we have
    $$(k_1+k_2-k_3)+\frac{2d-k_1-k_2}{q}\leq d-2+\frac{2d}{q}<(d-2)+2<q.$$
    Case 5B: $k_1+k_2\in \{d-1,d\}$. In this case $2d-k_1-k_2\leq d+1\leq q$. Therefore,
    $$(k_1+k_2-k_3)+\frac{2d-k_1-k_2}{q}\leq (d)+1\leq q.$$
    Case 5C: $k_1+k_2\geq d+1$. In this case we have $2d-k_1-k_2\leq d-1< q$, so $\frac{2d-k_1-k_2}{q}<1$. We want to show that $k_3\geq k_1+k_2-d$. Assume for the sake of contradiction that $k_3\leq k_1+k_2-d-1$. Now $c_d=c$, $c_{d-1}=\dots=c_{d-k_3}=c-1$ and $c_{d-k_3-1}\leq c-2$. Moreover, $d-k_3-1\geq 2d-k_1-k_2$. This means that
    \begin{align*}
    &\tbinom{c+d}{d} + \tbinom{c+d-1}{c}-\tbinom{c+2d-k_1-k_2}{c}+\tbinom{c+2d-k_1-k_2-1}{c-1}-1\\
    =&\tbinom{c+d}{d} +\sum_{j=2d-k_1-k_2+1}^{d-1} \tbinom{c-1+j}{j} +\sum_{j=1}^{2d-k_1-k_2} \tbinom{c-2+j}{j}\\
    &\geq \tbinom{c+d}{d} +\sum_{j=d-k_3}^{d-1} \tbinom{c-1+j}{j} +\sum_{j=1}^{d-k_3-1} \tbinom{c-2+j}{j}
    \geq \sum_{j=1}^{d}\tbinom{c_j+j}{j}
    =\sum_{j=1}^{d}\tbinom{a_j+j}{j}+ \sum_{j=1}^{d}\tbinom{b_j+j}{j}\\
    &\geq \sum_{j=d-k_1+1}^{d}\tbinom{c-1+j}{j} + \sum_{j=d-k_2+1}^{d}\tbinom{c-1+j}{j}
    =\tbinom{c+d}{c}-\tbinom{c+d-k_1}{c}+ \tbinom{c+d}{c}-\tbinom{c+d-k_2}{c}\\
    &\geq 2\tbinom{c+d}{c}- \tbinom{c+2d-k_1-k_2}{c}-\tbinom{c}{c}.
    \end{align*}
    This implies that
    $$\tbinom{c+d-1}{c}+\tbinom{c+2d-k_1-k_2-1}{c-1}
    \geq \tbinom{c+d}{c}.$$
    From here, we see that $\binom{c+2d-k_1-k_2-1}{c-1}\geq \binom{c+d-1}{c-1}$, that is, $k_1+k_2\leq d$. This contradicts the fact that we are in Case 5C. Therefore, $k_3\geq k_1+k_2-d$. We conclude that
    $$(k_1+k_2-k_3)+\frac{2d-k_1-k_2}{q}< (d)+1\leq q.$$
\end{itemize}
This completes the induction step with respect to $d$ and hence completes the proof.
\end{proof}

\begin{proof}[Proof of Lemma \ref{Lem: sum of Hrk}]
If $d=1$, then all $H_{r_k}(d-1,m-k+1)=0$, so there is nothing to prove. Now assume $d\geq 2$.
First, we want to show that if 
$w_{r_k}(d-1,m-k+1)=(a_1, \dots, a_{m-k+2})$, then $a_1= \dots= a_{l-k}=0$. Notice that the largest element (by lexicographical ordering) $(b_1, \dots, b_{m-k+2}) \in \Omega(d-1, m-k+1)$ with $b_1= \dots= b_{l-k}=0$ is the element $(0, \dots, 0, d-1, 0, \dots, 0)$, where $b_{l-k+1}=d-1$. By Lemma \ref{r as function of alpha}, this is the $1+\sum_{j=1}^{l-k}\binom{m+d-j-k}{d-2}$th largest element of $\Omega(d-1, m-k+1)$.
Therefore, showing that $a_1= \dots=a_{l-k}=0$ is equivalent to showing that $r_k > \sum_{j=1}^{l-k}\binom{m+d-j-k}{d-2}$.
Now, since $r_j \leq \binom{m+d-j}{d-1}$ and
\[\sum r_j >\tbinom{m+d}{d}-\tbinom{m+d+1-l}{d}=\sum_{k=1}^{l-1}\tbinom{m+d-k}{d-1},\]
we have
\begin{align*}
    r_k &=\sum_{j=1}^l r_j-\sum_{j=1,j \neq k}^lr_j
    >\sum_{j=1}^{l-1}\tbinom{m+d-j}{d-1}-\sum_{j=1,j \neq k}^l\tbinom{m+d-j}{d-1}\\
    &=\tbinom{m+d-k}{d-1}-\tbinom{m+d-l}{d-1}
    =\sum_{j=1}^{l-k}\tbinom{m+d-k-j}{d-2}.
\end{align*}
Thus, $w_{r_k}(d-1,m-k+1)$ has $a_1= \dots= a_{l-k}=0$. Consequently,
$$H_{r_k}(d-1,m-k+1)=H_{r_k-\sum_{j=1}^{l-k}\binom{m+d-k-j}{d-2}}(d-1,m-l+1).$$
Now, we let $s_k=r_k-\sum_{j=1}^{l-k}\tbinom{m+d-k-j}{d-2}$, so
$$0< s_k \leq \tbinom{m+d-k}{d-1}-\sum_{j=1}^{l-k}\tbinom{m+d-k-j}{d-2} =\tbinom{m+d-l}{d-1}.$$
Next, we also have
\begin{align*}
\sum_{k=1}^l s_k
&=\sum_{k=1}^l r_k-\sum_{k=1}^{l-1}\sum_{j=k+1}^l\tbinom{m+d-j}{d-2} 
=\sum_{k=1}^l r_k- \sum_{k=1}^{l-1} \Big( \tbinom{m+d-k}{d-1}-\tbinom{m+d-l}{d-1}\Big)\\
&=(l-1)\tbinom{m+d-l}{d-1}+\sum_{k=1}^l r_k -\sum_{k=1}^{l-1} \tbinom{m+d-k}{d-1}\\
&=(l-1)\tbinom{m+d-l}{d-1}+\sum_{k=1}^l r_k -\tbinom{m+d}{d} +\tbinom{m+d-l+1}{d}.
\end{align*}
Therefore, we have
$$(l-1)\tbinom{m+d-l}{d-1}<\sum_{k=1}^ls_k \leq l\tbinom{m+d-l}{d-1}.$$
Thus, by Lemma \ref{Lem: Sum of Hsk(d,m)}, we have
$$\sum_{k=1}^{l}H_{s_k}(d-1,m-l+1)\leq H_{\sum_{k=1}^l s_k-(l-1)\binom{m+d-l}{d-1}}(d-1,m-l+1).$$
But we just saw that $\sum_{k=1}^l s_k-(l-1)\binom{m+d-l}{d-1}=\sum_{k=1}^l r_k -\binom{m+d}{d} +\binom{m+d-l+1}{d}=r'$. This completes the proof.
\end{proof}

\subsection{Lemmas for Section \ref{sec: proof for incomplete gdc}}

In this subsection we prove several lemmas that were used to prove Theorem \ref{Thm l=1 with e}.

Recall from the introduction that $$\Omega(d,m)=\Big\{(\gamma_1,\dots,\gamma_{m+1})\in \N^{m+1}\mid \sum_{i=1}^{m+1} \gamma_i=d\Big\}.$$
and $\omega_r(d,m)=(\beta_1,\dots,\beta_{m+1})$ is the $r^{th}$ largest element under lexicographical ordering in $\Omega(d,m)$.
\begin{lemma}\label{bound on betaone}
Let $(\beta_1, \dots,\beta_{m+1})$ be the $r^{th}$ largest element of $\Omega(d-1,m)$. Then we have $\beta_1 \geq k$ if and only if $r \leq \binom{m+d-1-k}{d-1-k}$.
\end{lemma}
\begin{proof}
Denote $\beta=(\beta_1,\dots,\beta_{m+1})$ and $\gamma=(k,0,\dots,0,d-1-k)\in\Omega(d-1,m)$. Note that we have $\beta_1\geq k$ if and only if $\beta\geq_{\text{lex}} \gamma$. Further, notice that $\gamma$ is the $\binom{m+d-1-k}{d-1-k}^{th}$ largest element of $\Omega(d-1,m)$.
\end{proof}

By convention $\binom{m-1}{-1}=0$.

\begin{lemma}\label{Hrdm Hrdm-1}
Given $d,m\geq 1$, $0\leq e\leq d-2$, $\binom{m+e}{e}<r\leq \binom{m+d-1}{d-1}$ and $t\leq \binom{m+e-1}{e-1}$, we have
$$H_{r}(d-1,m)\geq q H_{r-t}(d-1,m-1).$$
\end{lemma}
\begin{proof}
Fix $d$. Recall that
$$\Omega(d-1,m)=\Big\{(\gamma_1,\dots,\gamma_{m+1})\in \N^{m+1}\mid \sum_{i=1}^{m+1} \gamma_i=d-1\Big\}.$$
It is ordered according to lexicographical ordering and its $r^{th}$ largest element is $w_r(d-1,m)$.
Consider the map $\phi_{m}:\Omega(d-1,m)\to \N$ defined by $\phi_m((\gamma_1,\dots,\gamma_{m+1}))=\sum_{i=1}^{m}\gamma_i q^{m-i}$. This means that $H_r(d-1,m)=\phi_m(w_r(d-1,m))$. In addition, note that the map $\phi_m$ preserves ordering.
Consider the map $\psi_{m}:\Omega(d-1,m-1)\to \Omega(d-1,m)$, given by $\psi_{m}(\gamma_1,\dots,\gamma_{m})=(\gamma_1,\dots,\gamma_m,0)$. This is an injective map, and it preserves ordering. Note that for $w\in \Omega(d-1,m-1)$, we have $\phi_m(\psi_m(w))\geq q \phi_{m-1}(w)$.

Now, let $s_0=\min\{s\geq r\mid w_s(d-1,m)\in \im(\psi_m)\}$. Suppose $w_{s_0}(d-1,m)=\psi_m(w_k(d-1,m-1))$. Then
$$s_0=k+|\{1\leq a\leq s_0\mid w_a(d-1,m)\notin \im(\psi_m)\}|.$$
However, by the definition of $s_0$, for $r \leq l \leq s_0-1$, we have $w_l(d-1,m)\notin \im(\psi_m)$. Thus,
\begin{align*}
&|\{1\leq a\leq s_0\mid w_a(d-1,m)\notin \im(\psi_m)\}|\\
&=s_0-r + |\{1\leq a< r\mid w_a(d-1,m)\notin \im(\psi_m)\}|\\
&\geq s_0-r+ |\{1\leq a \leq \tbinom{m+e}{e}\mid w_a(d-1,m)\notin \im(\psi_m)\}|.
\end{align*}

Now we want to count $|S|$ for 
$$S=\{w_a(d-1,m) \mid 1\leq a\leq \tbinom{m+e}{e}, w_a(d-1,m)\notin \im(\psi_m)\}.$$
By Lemma \ref{bound on betaone}, we know that $(\gamma_1,\dots,\gamma_{m+1}) \in S$ if and only if the following hold:
\begin{itemize}
    \item $\sum_{i=1}^{m+1}\gamma_i=d-1$;
    \item $a \leq \binom{m+e}{e}$. This happens if and only if $\gamma_1 \geq d-1-e$;
    \item $\gamma_{m+1} \geq 1$.
\end{itemize}
Such elements are given by $((d-1-e)+a_1,a_2 \dots,a_m, a_{m+1}+1)$, with $a_i \geq 0$ and $\sum_{i=1}^{m+1}a_i=e-1$. The number of solutions is $\binom{m+e-1}{e-1}$.
Thus, 
\begin{align*}
    s_0 \geq k+(s_0-r)+ |S|= k+(s_0-r)+ \tbinom{m+e-1}{e-1}
\end{align*}
Meaning, $k \leq r-\binom{m+e-1}{e-1}$. Since $t \leq \tbinom{m+e-1}{e-1}$, we have $k \leq r-t$. Thus, we have
\begin{align*}
    H_r(d-1,m) & \geq H_{s_0}(d-1,m) = \phi_{m}(w_{s_0}(d-1,m)) \\
    &=\phi_m(\psi_m(w_k(d-1,m-1)))
     \geq q\phi_{m-1}(w_k(d-1,m-1)) \\
    &=qH_k(d-1,m-1) \geq qH_{r-t}(d-1,m-1).\qedhere
\end{align*}
\end{proof}

\begin{lemma}\label{relation Hr}
If $\binom{m+e}{e} < r \leq \binom{m+e+1}{e+1}$, then
$$H_r(d-1,m)-H_{r-\binom{m+e}{e}}(d-1,m-1)=(d-2-e)q^{m-2}(q-1).$$
\end{lemma}
\begin{proof}
Let $w_r(d-1,m)=(\beta_1,\dots,\beta_{m+1})$. Since $\binom{m+e}{e}< r \leq \binom{m+e+1}{e+1}$, we have $\beta_1=d-2-e$. Denote $e_i=(0,\dots,0,1,0,\dots,0)$ with $1$ in $i^{th}$ spot.
Let $(i_1 ,i_2, \dots,i_{e+1})$ be the tuple for which $w_r(d-1,m)=(d-2-e)e_1+e_{i_1}+ \dots +e_{i_{e+1}}$ and $2 \leq i_1 \leq i_2 \leq \dots \leq i_{e+1} \leq m+1$. By definition of $w_r(d-1,m)$, there are $r$ elements of $\Omega(d-1,m)$ that are $\geq$ $w_r(d-1,m)$. Let
$$S_1=\{(j_1, \dots,j_{e+1}\mid  1 \leq j_1 \leq \dots \leq j_{e+1} \leq m+1, (j_1, \dots, j_{e+1} )\leq_{\text{lex}} (i_1, \dots, i_{e+1})\},$$
so $r=|S_1|$.

Next, denote $w_k(d-1,m-1)=(d-2-e)e_1+e_{i_1-1}+ \dots +e_{i_{e+1}-1}$. We want to show that
$r-k=\tbinom{m+e}{e}$.
As before let 
$$S_2=\{(a_1, \dots,a_{e+1}) \mid 1 \leq a_1 \leq \dots \leq a_{e+1} \leq m, (a_1, \dots, a_{e+1} )\leq_{\text{{lex}}} (i_1-1, \dots, i_{e+1}-1)\},$$
so $k=|S_2|$. 
Notice that we have an injection $f:S_2 \to S_1$ defined by
$$f(a_1, \dots,a_{e+1})=(a_1+1, \dots,a_{e+1}+1).$$
Thus, $r-k=|S_1\setminus f(S_2)|$. For $(j_1, \dots, j_{e+1}) \in S_1$, note that $(j_1, \dots, j_{e+1})\notin f(S_2)$ if and only if $j_1=1$. Therefore, since $i_2\geq 2$, we have 
\begin{align*}
&|S_1\setminus f(S_2)|\\
&=|\{(1,j_2 \dots,j_{e+1}\mid  1 \leq j_2 \leq \dots \leq j_{e+1} \leq m+1, (1,j_2, \dots, j_{e+1} )\leq_{\text{lex}} (i_1, \dots, i_{e+1})\}|\\
&=|\{(j_2, \dots,j_{e+1})\mid 1 \leq j_2 \leq \dots \leq j_{e+1} \leq m+1\}|
=\tbinom{m+e}{e}.
\end{align*}
Therefore, $k=r-\binom{m+e}{e}$. 
Finally we see that,
\begin{align*}
&H_r(d-1,m)-H_{r-\binom{m+e}{e}}(d-1,m-1)
=H_r(d-1,m)-H_{k}(d-1,m-1)\\
&=\Big((d-2-e)q^{m-1}+\sum_{u=1}^{e+1}\lfloor q^{m-i_u}\rfloor\Big) - \Big((d-2-e)q^{(m-1)-1}+\sum_{u=1}^{e+1}\lfloor q^{(m-1)-(i_u-1)}\rfloor\Big)\\
&=(d-2-e)(q^{m-1}-q^{m-2}).\qedhere
\end{align*}
\end{proof}

\begin{lemma}\label{Lem d-c to d}
Given $m\geq 1$, $1\leq c\leq d-1$ and $1\leq r\leq \binom{m+d-c}{d-c}$, we have
$$H_r(d-c,m)+cq^{m-1}\leq H_r(d,m).$$
\end{lemma}
\begin{proof}
Consider the map $f$ from $\Omega(d-c,m)$ to $\Omega(d,m)$, given by $f(\gamma_1,\dots,\gamma_{m+1})=(c+\gamma_1,\gamma_2,\dots,\gamma_{m+1})$. This map is injective and preserves order. Suppose $f(w_r(d-c,m))=w_k(d,m)$. Then $k$ is the number of elements of $\Omega(d,m)$ that are at least $w_k(d,m)$, so $k\geq r$. Therefore, we see that
\[H_r(d-c,m)+cq^{m-1}=H_k(d,m)\leq H_r(d,m).\qedhere\]
\end{proof}

\backmatter

\bmhead{Acknowledgments}
We thank Sudhir Ghorpade for introducing us to this problem. We thank Nathan Kaplan for many helpful discussions about the problem. The first author received support from the NSF Grant DMS 2154223.

\section*{Statements and Declarations}
Competing interests:
The authors have no competing interests.






\bibliography{sn-bibliography}

\end{document}